\documentclass[11pt,a4paper,reqno]{amsart}
\usepackage{amsmath,amssymb,amsfonts,epsfig,mathrsfs,cite}
\usepackage[T1]{fontenc}
\usepackage{color}
\usepackage{array}
\usepackage{amsthm}
\usepackage{amstext}
\usepackage{graphicx}
\usepackage{setspace}
\usepackage[margin=2.5cm]{geometry}
\usepackage{color}
\usepackage{enumitem}
\usepackage{undertilde}
\usepackage{xcolor}
\usepackage{amscd,psfrag}
\usepackage{yhmath}
\usepackage[mathscr]{eucal}

\usepackage{comment}

\allowdisplaybreaks[4]
\usepackage{slashed}

\makeatletter
\pdfpageheight\paperheight
\pdfpagewidth\paperwidth

\usepackage{epstopdf}
\usepackage{chngcntr}
\counterwithin{figure}{section}
\usepackage{mathrsfs}

\usepackage{indentfirst}	

\usepackage[normalem]{ulem}
\theoremstyle{plain}
\numberwithin{equation}{section}

\newtheorem{definition}{Definition}[section]
\newtheorem{theorem}[definition]{Theorem}
\newtheorem*{theorem*}{Theorem}

\newtheorem*{remark*}{Remark}
\newtheorem*{sideremark*}{Side Remark}

\newtheorem*{claim*}{Claim}
\newtheorem*{q*}{Question}
\newtheorem{lemma}[definition]{Lemma}
\newtheorem{corollary}[definition]{Corollary}
\newtheorem*{corollary*}{Corollary}

\newtheorem{proposition}[definition]{Proposition}

\newcommand{\R}{\mathbb{R}}
\newcommand{\mi}{{M^\varepsilon}}
\newcommand{\gi}{g^\varepsilon}
\newcommand{\ffi}{{F^\varepsilon}}
\newcommand{\map}{\rightarrow}

\newcommand{\xe}{{X_\varepsilon}}
\newcommand{\ye}{{Y_\varepsilon}}
\newcommand{\ffiu}{{(F^\varepsilon)^*}}
\newcommand{\G}{\Gamma}
\newcommand{\dd}{{\rm d}}
\newcommand{\na}{\nabla}
\newcommand{\nai}{\na^\varepsilon}
\newcommand{\euc}{{\mathfrak{e}}}
\newcommand{\phii}{{\Phi^{\varepsilon}}}
\newcommand{\barna}{\widetilde{\na}}
\newcommand{\comp}{{\phii\circ\ffi}}
\newcommand{\bi}{B^{\varepsilon}}
\newcommand{\ze}{{Z_\varepsilon}}
\newcommand{\we}{{W_\varepsilon}}
\newcommand{\emb}{\hookrightarrow}
\newcommand{\e}{\varepsilon}
\newcommand{\loc}{{\rm loc}}
\newcommand{\tM}{{\widetilde{M}}}
\newcommand{\tg}{{\tilde{g}}}
\newcommand{\Hom}{{\rm Hom}}
\newcommand{\so}{{\rm SO}}
\newcommand{\dist}{{\rm dist}}
\newcommand{\p}{\partial}

\newcommand{\RR}{{\mathscr{R}}}
\newcommand{\snu}{{\mathbb{S}^d}}
\newcommand{\can}{{\rm can}}
\newcommand{\cof}{{\rm cof}}
\newcommand{\dvg}{{\,\dd V_g}}
\newcommand{\tG}{{\widetilde{\Gamma}}}

\newcommand{\bigphi}{{\overline{\Phi}}}
\newcommand{\eucl}{{(\R^{d+1},\euc)}}
\newcommand{\nue}{{\mathbf \nu}^\varepsilon}
\newcommand{\nnue}{\underline{\nue}}
\newcommand{\hatgi}{{\widehat{g^\varepsilon}}}
\newcommand{\hatnai}{\widehat{\na^\varepsilon}}
\newcommand{\hatbi}{\widehat{B^\varepsilon}}
\newcommand{\push}{{(\comp)_*}}
\newcommand{\weak}{\rightharpoonup}
\newcommand{\curve}{{\widehat{R^\varepsilon}}}
\newcommand{\hatri}{{\widehat{R^\varepsilon}}}
\newcommand{\two}{{\rm II}}
\newcommand{\hgai}{{\widehat{\Gamma^\varepsilon}}}
\newcommand{\xxe}{\underline{{X_\varepsilon}}}
\newcommand{\yye}{\underline{{Y_\varepsilon}}}
\newcommand{\dve}{\dd V_\euc}
\newcommand{\mres}{\mathbin{\vrule height 1.6ex depth 0pt width
0.13ex\vrule height 0.13ex depth 0pt width 1.3ex}}

\title{On Asymptotic Rigidity and Continuity Problems in Nonlinear Elasticity on Manifolds and Hypersurfaces}
\author{Gui-Qiang G. Chen}
\address{G.-Q. Chen: Mathematical Institute,\
 University of Oxford, Oxford, OX2 6GG, UK}
\email{\texttt{chengq@maths.ox.ac.uk}}
\author{Siran Li}

\address{Siran Li: School of Mathematical Sciences $\&$ Key Laboratory of Scientific and Engineering Computing (Ministry of Education), Shanghai Jiao Tong University, No.~6 Science Buildings,
800 Dongchuan Road, Minhang District, Shanghai, China (200240); and New York University -- Shanghai, Office 1146, 1555 Century Avenue, Pudong District, Shanghai, China (200122); and  NYU-ECNU Institute of Mathematical Sciences, Room 340, Geography Building, 3663 North Zhongshan Road, Shanghai, China (200062)}

\email{\texttt{siran.li@sjtu.edu.cn}}

\author{Marshall Slemrod}
\address{M. Slemrod: Department of Mathematics, University of Wisconsin, Madison, WI 53706, USA}
\email{\texttt{slemrod@math.wisc.edu}}


\date{\today}

\begin{document}

\begin{abstract}
Intrinsic nonlinear elasticity deals with the deformations of elastic bodies as isometric immersions of Riemannian manifolds into the
Euclidean spaces (see Ciarlet \cite{ciarlet, ciarlet-new}).
In this paper, we study the rigidity and continuity properties of elastic bodies for the intrinsic approach
to nonlinear elasticity.
We first establish a geometric rigidity estimate for mappings from Riemannian manifolds to spheres
(in the spirit of Friesecke--James--M\"{u}ller \cite{fjm}),
which is the first result of this type for the non-Euclidean case as far as we know.
Then we prove the asymptotic rigidity of elastic membranes under suitable geometric conditions.
Finally, we provide a simplified geometric proof of the continuous dependence of deformations of
elastic bodies on the Cauchy--Green tensors and second fundamental forms, which extends
the Ciarlet--Mardare theorem in \cite{cm3} to arbitrary dimensions and co-dimensions.\medskip\\
\medskip
{\bf R\'{e}sum\'{e}}\\
L'\'{e}lasticit\'{e} non-lin\'{e}aire intrins\`{e}que consid\`{e}re les d\'{e}formations de corps \'{e}lastiques comme des immersions isom\'{e}triques de vari\'{e}t\'{e}s Riemanniennes dans l'espace Euclidien (voir Ciarlet \cite{ciarlet, ciarlet-new}). Dans cet article, en suivant l'approche intrins\`{e}que de l'\'{e}lasticit\'{e} non-lin\'{e}aire, nous \'{e}tudions les propri\'{e}t\'{e}s de rigidit\'{e} et de continuit\'{e} de corps \'{e}lastiques. Premi\`{e}rement, nous prouvons une estim\'{e}e de rigidit\'{e} g\'{e}om\`{e}trique pour les applications de vari\'{e}t\'{e}s Riemanniennes aux sph\`{e}res (dans l'esprit de Friesecke--James--M\"uller \cite{fjm}), qui est \`{a} notre connaissance le premier r\'{e}sultat de ce type dans le cas non-Euclidien. Ensuite, nous prouvons la rigidit\'{e} asymptotique de membranes \'{e}lastiques sous des hypoth\`{e}ses g\'{e}om\'{e}triques appropri\'{e}es. Enfin, nous donnons une preuve g\'{e}om\'{e}trique simplifi\'{e}e de la d\'{e}pendence continue des d\'{e}formations de corps \'{e}lastiques par rapport \`{a} leur tenseur de Cauchy--Green et leur seconde forme fondamentale, ce qui \'{e}tend le th\'{e}or\`{e}me de Ciarlet--Mardare \cite{cm3} en dimensions et codimensions arbitraires.
\end{abstract}
\maketitle




\section{Introduction}
\label{sec: intro}

Nonlinear elasticity has long been an important subject in mathematics, physics, and engineering.
One of the major objectives of elasticity theory is to determine the {\em deformation} undergone
by the elastic bodies in response to external forces
and boundary conditions ({\it cf}. \cite{ciarlet, cm-new}).
For an elastic body modelled as a $2$- or $3$-dimensional isometrically immersed submanifold of the Euclidean space $\R^3$,
its deformation $\Phi$ is an isometric immersion.
The {\em intrinsic approach} to nonlinear elasticity
recasts the problems concerning deformation $\Phi$ to those concerning the {\em Cauchy--Green tensor},
{\it i.e.}, the Riemannian metric determined by $\Phi$
(see Antman \cite{a}, Ciarlet \cite{ciarlet, ciarlet-new}, Ciarlet--Mardare \cite{cm2, cm2010}, and the references cited therein).

The intrinsic approach to nonlinear elasticity is based on the {\em fundamental theorem of surface theory}.
It says that, under suitable topological and regularity conditions,
deformation $\Phi$
of a surface $M$ isometrically immersed in $\R^3$ can be recovered
from the Riemannian metric $g$ ({\it i.e.}, the Cauchy--Green tensor)
and the {\em second fundamental form} $B$.
Tensor $B$ is determined by the manner in which $M$ is immersed in the ambient space $\R^3$.
In differential geometry, one says that $g$ determines the {\em intrinsic geometry} of $M$,
and $B$ the {\em extrinsic geometry}.
The literature on this topic is abundant: we refer to \cite{ciarlet, ciarlet-new, cgm, cl, cm1} for the fundamental theorem
of surface theory, to \cite{m03, m05, m07} for generalizations to surfaces with lower regularity ({\it i.e.}, $W^{2,p}$),
and to \cite{chenli,s,t} for generalizations to arbitrary dimensions and co-dimensions, among others.

Throughout this paper, {\em elastic bodies} are modelled by Riemannian manifolds isometrically immersed in the Euclidean spaces,
which are said to have lower regularity if the associated isometric immersions have only $W^{2,p}$--regularity ({\it cf}. \cite{cm3,m05,chenli}).
An {\em elastic membrane} is an elastic body with co-dimension $1$, {\it i.e.}, a $d$-dimensional hypersurface isometrically
immersed in $\R^{d+1}$. Case $d=2$ is the situation of physical relevance,
though we investigate for the general dimension case $d\ge 2$ mathematically.
More precisely, we address three problems concerning the rigidity and continuity properties of elastic bodies.

We first investigate the Riemannian analogues of the
geometric rigidity estimate due to Friesecke--James--M\"{u}ller  \cite{fjm}:
If the gradient of a vector field
on a Euclidean Lipschitz domain
is close to a Euclidean rigid motion on average, then it is close to a specific rigid motion;
see Proposition \ref{propn: geometric rigidity} below for the precise statement.
It is a quantitative version of the rigidity theorem due to Re\u{s}etnjak \cite{r, r2}.
It remains an open question whether the analogous results for mappings between Riemannian manifolds remain valid;
{\it cf.} Kupferman--Maor--Shachar \cite{kms}.
In this paper, as a first step towards the above open question,
we establish a geometric rigidity estimate
from $d$-dimensional Riemannian manifolds to spheres of dimension $d\ge 2$.
To achieve this, we exploit the Riemannian Piola identity established in \cite{ks}
and several ideas for dealing with harmonic maps to tackle the nonlinearities,
with crucial use of the special geometry of spheres.
We remark in passing that other rigidity results (\emph{e.g.},
infinitesimal rigidity and dimension reduction) have been established
in the literature; see \cite{new-lmp,new-hlp,new-lmp',new-l} and the references cited therein.

We then consider a family of elastic membranes $\{(\mi, \gi)\}_{\varepsilon>0}$
in $\R^{d+1}$ via $W^{2,p}$--isometric immersions $\{\Phi^\varepsilon\}_{\varepsilon>0}$
with bi-Lipschitz homeomorphisms $\ffi : (M,g) \map (\mi, \gi)$ for a fixed Riemannian manifold $(M,g)$.
The {\em asymptotic problem} is whether it is possible to extract
a subsequence of $\{\Phi^\varepsilon \circ \ffi\}_{\varepsilon>0}$ that converges weakly
to an isometric immersion of $(M,g)$ with the corresponding extrinsic geometries, {\it i.e.}, second fundamental forms,
provided that $\ffi$ are ``asymptotically isometric'' and $\Phi^\varepsilon$ are uniformly bounded in $W^{2,p}$.
Our second result
addresses the asymptotic problem by establishing the convergence of extrinsic
geometries under natural geometric assumptions; also see the recent work by Alpern--Kupferman--Maor \cite{new-akm}.
Our approach to the asymptotic problem is based on the weak continuity of the {\em Gauss--Codazzi equations} for isometric immersions.
The Gauss--Codazzi equations are a first-order nonlinear system of partial differential equations (PDEs)
in terms of the second fundamental form.
Its weak continuity has been established in \cite{chenli,chenli2,csw1,csw2}
by first observing its intrinsic div-curl structure and developing
a global {\em compensated compactness} approach.
Based on these developments, we answer the asymptotic rigidity problem in the affirmative,
under the assumption that $F^\varepsilon_* g- \gi$  converges to zero in suitable Sobolev norms
where $F^\varepsilon_*$ denotes  the pushforward under $\ffi$.

Finally, we analyze the question of continuous dependence  of the isometric immersion $\Phi$ in the $W^{2,p}$--norm
with respect to the $W^{1,p}$--norm of metric $g$ and the $L^p$--norm of the second fundamental form $B$.
In Ciarlet--Mardare \cite{cm3}, the continuous dependence result
for $M$ as a simply-connected, open, bounded subsets of $\R^2$ with Lipschitz boundary was obtained.
This has been further generalized to other Fr\'{e}chet topologies in \cite{new-cmm}.
We provide here a simplified geometric proof,
which also applies to isometric immersions of simply-connected Riemannian manifolds with arbitrary dimensions
and co-dimensions.
The proof is based on the arguments in our proof of the realisation theorem \cite[Theorem 5.1]{chenli},
as well as the {\em Cartan structural equations}
for isometric immersions (see \cite{c}) and the analytic lemmas due to Mardare \cite{m03, m05, m07}.

The rest of the paper is organized as follows:
In $\S \ref{sec: geometry}$, we introduce some notations and present some basics of differential geometry and non-Euclidean elasticity
needed for subsequent developments.
In $\S \ref{sec: geom rigid}$, we establish the geometric rigidity estimate theorem, Theorem \ref{thm: geom rigidity},
for mappings from $d$-dimensional Riemannian
manifolds to spheres of dimension $d\ge 2$, which extends the results of Friesecke--James--M\"{u}ller in the Euclidean spaces in \cite{fjm}.
In $\S \ref{sec: convergence}$, we establish the asymptotic rigidity of elastic bodies, Theorem \ref{thm: convergence}.
In $\S \ref{sec: continuous dep}$, we provide a simplified proof of the continuous dependence of the deformations on the Cauchy--Green tensor
and extrinsic geometry, Theorem \ref{thm: continuous dependence}.

\section{Geometry of Elastic Bodies and Curvatures}
\label{sec: geometry}

In this section, we introduce some notations and present some basics of differential geometry
and non-Euclidean elasticity needed for subsequent development.

\subsection{Riemannian Submanifold Theory}\label{subsec: geom prelim}
Let $(M,g)$ and $\{(\mi, \gi)\}_{\varepsilon>0}$ be Riemannian manifolds, {\it a.k.a.} elastic bodies.
Let $\ffi: M \map \mi$ be a bi-Lipschitz homeomorphism for each $\varepsilon>0$. The pushforward $F^\varepsilon_* g$ defines another metric on $\mi$:
\begin{equation*}
F^\varepsilon_* g(\xe,\ye) := g(\ffiu \xe, \ffiu \ye) \qquad \text{ for $\xe,\ye \in \G(T\mi)$},
\end{equation*}
where $\ffiu \xe:=\dd(\ffi)^{-1}\xe$ is the pullback vector field of $\xe$ which is
well-defined as $\ffi$ is a bi-Lipschitz homeomorphism (see also \S 2.2 below),
$T\mi$ is the tangent bundle of $\mi$, and $\G(T\mi)$ is the space of tangential
vector fields to $\mi$ (similar for $M$).

Let $\na$ and $\nai$ be the Levi-Civita connections on $(M,g)$ and $(\mi, \gi)$, respectively.
Then $F^\varepsilon_*\na$ defines another affine connection on $\mi$, known as the pushforward connection:
\begin{equation*}
[F^\varepsilon_*\na]_\xe \ye := \na_{\ffiu \xe} \,(\ffiu \ye)  \qquad \text{ for $\xe,\ye \in \G(T\mi)$}.
\end{equation*}
Unless $\ffi$ is an isometry,  $F^\varepsilon_*\na$ and $\gi$ are unrelated in general.

Consider an isometric immersion $\Phi^\varepsilon : \mi \map \R^{D}$, where $\dim \mi=\dim M = d$,
and $\R^{D}$ is equipped with the Euclidean metric $\euc$.
By definition, differential $\dd\Phi^\varepsilon$ is everywhere injective, and $(\Phi^\varepsilon)^*\euc = \gi$.
It defines the second fundamental form:
$$
B^\varepsilon: \G(T\mi) \times \G(T\mi) \map \G((T\mi)^\perp),
$$
where $(T\mi)^\perp$ is the normal bundle of $\Phi^\varepsilon$:
\begin{equation*}
(T\mi)^\perp := T\R^D/T(\phii(\mi)).
\end{equation*}
That is, $(T\mi)^\perp$ is the quotient bundle of two tangent bundles --- for each $x\in\mi$,
its fiber $(T_x\mi)^\perp$ is the quotient vector space $T_{\phii(x)}\R^D \slash T_{\phii(x)}(\phii(\mi))$.
More precisely, if $\barna$ denotes the Levi-Civita connection on $\R^D$, then
\begin{equation}\label{2.1a}
B^\varepsilon (\xe,\ye) := \barna_{\dd\phii(\xe)}{\dd\phii(\ye)} - \nai_\xe \ye \qquad \text{ for } \xe,\ye \in \G(T\mi).
\end{equation}
The right-hand side of \eqref{2.1a} can be understood as follows: we can locally extend $\dd\phii(\xe)$ and $\dd\phii(\ye)$,
which are tangential vector fields on $\comp(M)\subset \R^D$, to vector fields $X'_\varepsilon$ and $Y'_\varepsilon$ on $\R^D$, respectively,
and then set
$$
\bi(\xe,\ye):= \barna_{X'_\varepsilon} Y'_\varepsilon-\nai_{X'_\varepsilon}Y'_\varepsilon.
$$
This is independent of the choice of extensions $X'_\varepsilon$ and $Y'_\varepsilon$; see do Carmo \cite[pp.126--127, \S 6]{doc}.

A compatibility condition for the isometric immersion $\Phi^\varepsilon: \mi \map \R^{D}$
is the Gauss--Codazzi equations (GCE).
They arise from the compatibility of curvatures: the flat curvature of $\R^{D}$
decomposes along $T\mi$ and $(T\mi)^\perp$.
Denote by $R^\varepsilon: \G((TM^\varepsilon)^{\otimes 4})\map\R$ the Riemann curvature tensor of $(\mi,\gi,\nai)$.
The Gauss equation reads
\begin{equation}\label{gauss}
\mathfrak{e}(B^\varepsilon(\xe,\ze), B^\varepsilon(\ye,\we))
-\mathfrak{e}(B^\varepsilon(\xe,\we), B^\varepsilon(\ye,\ze)) = R^\varepsilon(\xe,\ye,\ze,\we),
\end{equation}
and the Codazzi equation reads
\begin{equation}\label{codazzi}
\barna_{\dd\phii(\xe)} B^\varepsilon(\ye,\ze) = \barna_{\dd\phii(\ye)} B^\varepsilon(\xe,\ze),
\end{equation}
for $\xe,\ye,\ze,\we \in \G(T\mi)$.

If $\mi \emb \R^{D}$ are not hypersurfaces, {\it i.e.},
the normal bundle $(T\mi)^\perp$ has rank greater than $1$, then there is an additional compatibility equation named after Ricci:
\begin{equation}\label{ricci}
\gi( [S^\varepsilon_{\xi_\varepsilon}, S^\varepsilon_{\eta_\varepsilon}]\xe,\ye) = R^{\varepsilon,\perp} (\xe,\ye,\eta_\varepsilon,\xi_\varepsilon).
\end{equation}
Here $S^\varepsilon$ is the shape operator of $B^\varepsilon$ (which is equivalent to $B^\varepsilon$) and $R^{\varepsilon,\perp}$ is the Riemann curvature
of bundle $(T\mi)^\perp$ with respect to $(\gi, \nai)$:
\begin{align*}
&R^{\varepsilon,\perp}(\xe,\ye,\eta_\varepsilon,\xi_\varepsilon)\\
&:= \mathfrak{e}
(\nabla^{\varepsilon,\perp}_{\dd\phii(\xe)}\nabla^{\varepsilon,\perp}_{\dd\phii(\ye)} \eta_\varepsilon
- \nabla^{\varepsilon,\perp}_{\dd\phii(\ye)}\nabla^{\varepsilon,\perp}_{\dd\phii(\xe)} \eta_\varepsilon
+\nabla^{\varepsilon,\perp}_{[\dd\phii(\xe),\dd\phii(\ye)]} \eta_\varepsilon,\,\xi_\varepsilon).
\end{align*}
Connection $\nabla^{\varepsilon,\perp}$ is the orthogonal projection of $\barna$ onto $(T\mi)^\perp$.
The vector fields $\xe,\ye \in \G(T\mi)$ and $\eta_\varepsilon,\xi_\varepsilon \in \G((T\mi)^\perp)$ are arbitrary.
In the weak regularity case that $\Phi^\varepsilon \in W^{2,p}$,
solutions of \eqref{gauss}--\eqref{ricci}
are understood in the sense of distributions,
and the injectivity of $\dd\Phi^\varepsilon$ is defined for a continuous representative
in the Sobolev class.

Recognizing the underlying ``div-curl structure''
and developing the geometric compensated compactness argument,
we have established a weak continuity theorem
of the Gauss--Codazzi--Ricci equations (GCRE) \eqref{gauss}--\eqref{ricci} in \cite{chenli} (also see \cite{csw2}):

\begin{proposition}[{\cite[Theorem $4.1$]{chenli}}]\label{propn: chenli, weak continuity thm for GCR eq}
Let $M$ be a Riemannian manifold with $W^{1,p}\cap L^\infty$--metric $g$ for $p>2$.
Assume that the second fundamental forms and normal connections $\{(B^\e, \na^{\e,\perp})\}_{\e>0}$
are solutions of the Gauss--Codazzi--Ricci equations \eqref{gauss}--\eqref{ricci}
and have a uniform $L^p_\loc$--bound
on the Riemannian manifold $(M,g)$.
Then, after passing to a subsequence if necessary,
$(B^\e, \na^{\e,\perp})$ converges weakly in $L^p_\loc$ to a solution
of the Gauss--Codazzi--Ricci equations \eqref{gauss}--\eqref{ricci}.
\end{proposition}

Its geometric analogue is the weak rigidity theorem of isometric immersions below.
{Throughout this paper,  an immersion $\Phi$ is understood in the sense that
$D\Phi$ has nonzero determinant {\it a.e.}. Note that $D\Phi$ is well-defined {\it a.e.}
for $\Phi\in W^{2,p}$}.

\begin{proposition}[{\cite[Corollary $5.2$]{chenli}}]\label{propn: chenli, weak rigid thm of isom imm}
Let $M$ be a $d$-dimensional simply-connected Riemannian manifold with $W^{1,p}$--metric $g$ for $p>d$.
Assume that $\{\Phi^\e\}_{\e>0}$ is a family of isometric immersions of $(M,g)$ into a Euclidean space, uniformly bounded in $W^{2,p}_\loc$,
whose second fundamental forms and normal connections are $(B^\e, \na^{\e,\perp})$.
Then, after passing to a subsequence if necessary, $\Phi^\e$ converges weakly in $W^{2,p}_\loc$
to an isometric immersion $\Phi$ of $(M,g)$ whose second fundamental form and normal connection are the weak $L^p_\loc$--limits
of $(B^\e,\na^{\e,\perp})$,
obeying the Gauss--Codazzi--Ricci equations \eqref{gauss}--\eqref{ricci}.
\end{proposition}

Propositions \ref{propn: chenli, weak continuity thm for GCR eq}--\ref{propn: chenli, weak rigid thm of isom imm}
serve as a starting point for $\S\ref{sec: convergence}$--$\S\ref{sec: continuous dep}$.
In the proof of the asymptotic rigidity theorem, Theorem \ref{thm: convergence} below,
we employ a variant of
Proposition \ref{propn: chenli, weak continuity thm for GCR eq}.
Moreover, the continuous dependence theorem, Theorem \ref{thm: continuous dependence}, enables us
to obtain a stronger version of  Proposition \ref{propn: chenli, weak rigid thm of isom imm}.

We also introduce two notational conventions.
First, the Einstein summation convention is used throughout: the repeated lower and upper indices are understood to be summed over.
Second, for two tensor fields $S$ and $S'$ on the same manifold, $S \star S'$ denotes a generic linear combination
of quadratic expressions for the components of $S$ and $S'$.

\subsection{Vector Bundles}\label{subsec: vector bdls}
Let $f: (M,g) \map (\tM, \tg)$ be a mapping between two Riemannian manifolds.
Then its differential $\dd f: TM \map T\tM$ can be viewed as a section of the vector bundle $T^*M \otimes f^*T\tM \cong \Hom(TM;f^*T\tM)$.
Here and hereafter, for a vector bundle $E$ over manifold $\tM$, $f^*E$ is the {\em pull-back bundle} over $M$.

{
In the above setting, we define as in \cite[\S 1.1, p.369]{kms}:
\begin{align}\label{new_SO,def}
\so(g,\tg) := \Big\{ \xi\,:\, &\,TM \map T\tM \,\,\, \text{such that, for each $x\in M$, } \nonumber\\
        &\,\,\xi_x:T_xM\to T_{\xi'(x)}\tM \text{ is an  orientation-preserving} \nonumber\\
        &\,\text{ isometry with respect to $(g, \tg)$} \Big\},
\end{align}
where $\xi':M\to\tM$ is the map between the base points associated to $\xi$.
That is, for $\xi(F_1)=F_2$,
\begin{equation}\label{new_Q'}
\xi'(\pi_M(F_1)) := \pi_{\tM}(F_2),
\end{equation}
where $F_1 \in TM$ and $F_2\in T\tM$ are fibers, and $\pi_M: TM \to M$ and $\pi_{\tM}:T\tM \to \tM$ are natural
projections onto the base points.
}

{As a side remark, given $f: M\to\tM$,  $\dd f \in \so(g,\tg)$ is systematically written
as $\dd f \in \so(g;f^*\tg)$ in \cite{kms}.
However, we adhere to the notation, $\so(g,\tg)$, which will be more convenient
for our purpose (in particular, for the formulation of Theorem~\ref{thm: geom rigidity}).
}

In $\S \ref{sec: geom rigid}$ below, we are interested in the distance between a given matrix field over $M$ and $\so(g,\tg)$.
More precisely, for $Q \in \G(T^*M \otimes T\tM)$ (that is, for each $x \in M$, $Q(x)$ is a linear
homomorphism, {\it i.e.}, a $(d\times d)$--matrix, from $T_xM$ to $T_{Q'(x)}\tM$), we consider the map on $M$:
\begin{equation}
\dist(Q, \so(g,\tg))\,:\,\,\, x \,\longmapsto \, \dist(Q'(x), \so(g,\tg)|_{x; Q}),
\end{equation}
where $Q'$ is as in \eqref{new_Q'} and
\begin{align*}
\so(g,\tg)|_{x; Q} :=\big\{S:\, &\, T_xM\to T_{Q'(x)}\tM \,\,\, \text{is orientation-preserving}\\
  &\, \text{ with $S^*(\tg|_{Q'(x)}) = g|_x$}\big\}.
\end{align*}

We need some more notations: Let $E$ be a vector bundle over $(M,g)$ with
the bundle metric $g^E$.
The natural metric on $\G(T^*M\otimes E) \equiv \Omega^1(M,E)$ is the product metric of $g$ and $g^E$, denoted by $g\otimes g^E$.
For notational convenience, we sometimes write $[g\otimes g^E](\bullet,\bullet)\equiv \langle\bullet,\bullet\rangle_{g\otimes g^E}$,
and similarly for the other metrics.
Moreover, $\G_0(E)$ denotes the space of sections $\sigma: M \map E$ such that $\sigma|_{\p M}=0$,
and $\G_{\rm c}(E)$ denotes the space of compactly supported sections. For a manifold ${M}$,
we reserve symbol $\na^{{M}}$ for the Levi-Civita connection on ${M}$.

\section{Geometric Rigidity for Mappings from Riemannian Manifolds to Spheres}
\label{sec: geom rigid}

A geometric rigidity estimate was first established in Friesecke--James--M\"{u}ller \cite[\S 4]{fjm} as stated below.
Roughly speaking, if the gradient
of a vector field $v$ is close to a Euclidean rigid motion on average, then it is indeed close to a specific rigid motion.

\begin{proposition}[\cite{fjm},
Theorem~3.1 and the ensuing comment]\label{propn: geometric rigidity}
Let $\Omega \subset \R^d$ be a bounded Lipschitz domain with $d\geq 2$, and let $1<p<\infty$.
Then there is $C=C(\Omega,p)>0$ so that,
for each $v \in W^{1,p}(\Omega,\R^d)$, there exists $\RR \in \so(d)$ such that
\begin{equation}\label{3.1a}
\|\na v - \RR\|_{L^p(\Omega)} \leq C\, \|\dist(\na v, \so(d))\|_{L^p(\Omega)},
\end{equation}
where $\so(d)$ is the group of all rotations about the origin in $\R^d$ under the operation of composition.
\end{proposition}

The above result can be viewed as a quantitative version of the Re\u{s}etnjak rigidity theorem
in nonlinear elasticity in \cite{r, r2}.
In the recent work by Kupferman--Maor--Shachar \cite{kms},
the Re\u{s}etnjak rigidity theorem was generalized to
Lipschitz maps
{$f: M \map \tM$} between Riemannian manifolds.
It is raised as an open question in \cite{kms} whether the quantitative inequality \eqref{3.1a}
admits generalizations
to mappings (not necessarily Lipschitz) between Riemannian manifolds.

The difficulty of the above question lies in the nonlinearity induced by the geometry of Riemannian manifolds.
Indeed, the proof of Proposition \ref{propn: geometric rigidity} by Friesecke--James--M\"{u}ller \cite{fjm}
relies essentially on the flat geometry of $\R^d$, especially the interior {\it a priori} estimates for harmonic functions on $\R^d$.
In contrast, in the Riemannian setting, the role of harmonic functions is, roughly speaking,
played by {\em harmonic maps}.
The PDEs for harmonic maps are a quasilinear elliptic system,
for which the desired regularity theory and {\it a priori}
estimates are largely missing.

In this section, we present a
{variant} of Proposition \ref{propn: geometric rigidity} for
mappings $f: M\map \tM$
{under the restrictions both that $p=2$ and that $f$ satisfies the higher regularity assumptions}.
Throughout this section, $M$ is a Riemannian manifold, and $\snu=\{x=(x^1,\cdots, x^{d+1})\in \R^{d+1}\,:\, |x|=1\}$
is the unit $d$-sphere equipped with the {\em round metric can} parameterized
by $d$-angles $(\phi_1, \cdots, \phi_d)$:
\begin{equation*}
{\rm can} = \dd\phi_1 \otimes \dd\phi_1 + \sum_{i=2}^d \big(\prod_{j=1}^{i-1} \sin^2\phi_j \big) \dd\phi_i\otimes \dd\phi_i,
\end{equation*}
so that $x\in \snu$ is represented by
\begin{align*}
x^1= \cos\phi_1,\quad\,
x^j = \cos\phi_j\prod_{i=1}^{j-1} \sin\phi_i\,\,\,\, \text{for $2\leq j \leq d$}, \quad\,
x^{d+1} = \prod_{i=1}^d \sin\phi_i.
\end{align*}
The key ingredients of the proof include the specific structures of $\snu$-valued harmonic map equations (see \cite{es, hw}),
as well as the Riemannian Piola identity established by Kupferman--Maor--Shachar (see \cite[Theorem $2$]{kms} and \cite{ks}).

{Here and hereafter in this section, an element $\RR$ in $\so(g,\can)$ is said to be a rigid motion
if it equals to a constant orthogonal $(d+1)\times (d+1)$-matrix restricted onto $\snu\subset \R^{d+1}$.}

\begin{theorem}\label{thm: geom rigidity}
Let $d \in \mathbb{Z}_{\geq 2}$ $($i.e., $d\ge 2$ integers$)$, and let $\mathbf{n}(d)$ be any number strictly greater than $\frac{d}{2}$.
Let $(M,g)$ be a Riemannian manifold that can be $C^{{\mathbf{n}(d)}+2}$-isometrically immersed into
the round sphere $(\snu,can)$, with possibly non-empty boundary $\p M$. Then, for each $f\in W^{\mathbf{n}(d)+1,\infty}_0(M,g; \snu,\can)$, there exists a rigid motion $\RR$ such that
\begin{equation}\label{3.1b}
\|\dd f - \RR\|_{L^2(M,g;\,\snu,\can)} \leq C \|\dist(\dd f, SO(g, \can))\|_{L^2(M,g)},
\end{equation}
where $C>0$ depends only on $\|f\|_{W^{\mathbf{n}(d)+1,\infty}(M,g;\,\snu,\can)}$ and the $C^{\mathbf{n}(d)+2}$-geometry of $(M,g)$.
\end{theorem}

The assumption that $\|f\|_{W^{\mathbf{n}(d)+1,\infty}(M,g;\,\snu,\can)}<\infty$ is not sharp;
in fact, our goal here is not to find the sharp conditions on $f$.
For example, the assumption that $\|f\|_{W^{2,\infty}\cap W^{3,2}(M,g;\,\snu,\can)}<\infty$ will suffice when $d\in\{2,3\}$.

By the dilation: $can \mapsto \lambda^2 can$, $\lambda \neq 0$, the theorem holds when the target manifold is replaced
by any round $d$-sphere, with constant $C$ depending additionally on $\lambda$.
The theorem is also invariant under the transform: $\dd f \mapsto O \circ \dd f$
for a rigid motion $O$, and inequality~\eqref{3.1b} holds with the same constant $C$.

In contrast to the Euclidean case (Proposition \ref{propn: geometric rigidity}),
we impose the higher integrability assumption on $f$ and obtain estimate \eqref{3.1b} for the non-Euclidean case
in Theorem \ref{thm: geom rigidity}.
It would be interesting to see whether the result can be improved
by finding a uniform constant $C$ (depending only on the geometry of $M$ and $\snu$).

To understand Theorem~3.2, we now present the following variant
of an asymptotic rigidity result that was established for the more general case of nearly conformal maps
by Re\u{s}etnjak \cite{r}; see also  \cite[Corollary~3.3]{fjm}.

\begin{corollary}\label{cor: resetnjak}
{Let $d \in \mathbb{Z}_{\geq 2}$, and let $\mathbf{n}(d)$ be any number strictly greater than $\frac{d}{2}$.}
Let $(M,g)$ be a Riemannian manifold that can be
{$C^{{\mathbf{n}(d)}+2}$}-isometrically immersed in $(\snu,\can)$.
Let $\{f_\varepsilon\}$ be a family of mappings from $M$ to $\snu$ with a uniform bound in
{$W^{\mathbf{n}(d)+1,\infty}_0(M,g;\,\snu,\can)$}.
If the $L^2$-norm of the distance between $\dd f_\varepsilon$ and the group of orientation-preserving isometries from $(M,g)$ to $(\snu,can)$ shrinks to zero at a rate of $\mathcal{O}(\e)$,
then there exists a particular rigid motion whose $L^2$-distance to $\dd f_\varepsilon$ shrinks to zero also at a rate of $\mathcal{O}(\e)$, after passing to subsequences.
\end{corollary}
\begin{proof}
By assumption, $\{\dd f_\varepsilon\}$ is uniformly bounded in
{$W^{\mathbf{n}(d),\infty}(M,g;\,\snu,\can)$}.
Hence, thanks to a standard compactness argument, it contains a subsequence $\{\dd f_{\varepsilon_j}\}$
that converges to some
{$F \in W^{\mathbf{n}(d),\infty}(M,g;\,\snu,\can)$} strongly in $L^2$.
Since $\|{\rm dist}(\dd f_\varepsilon, \so(g,\can))\|_{L^2}\to 0$ at rate $\mathcal{O}(\e)$,
it follows that $F\in \so(g,\can)$ \emph{a.e.}.
On the other hand, by Theorem~\ref{thm: geom rigidity}, for each $j$,
there is a rigid motion $\RR_{\varepsilon_j}$ with $\|\dd f_{\varepsilon_j} - \RR_{\varepsilon_j}\|_{L^2} \to 0$
at rate $\mathcal{O}(\e_j)$.
Since the group of rigid motions is compact,
there is a further subsequence $\{\varepsilon_{j_k}\}\subset \{\varepsilon_{j}\}$
such that $\|\RR_{\varepsilon_{j_k}}-\RR_0\|_{L^2} \to 0$ for a constant rigid motion $\RR_0$.
Then it follows from the  uniqueness of limits that $F = \RR_0$ \emph{a.e.}.
\end{proof}

\medskip
To prove the geometric rigidity theorem, Theorem 3.2,
we need the Piola identity in terms of the extrinsic geometry.

\subsection{The Riemannian Piola Identity}
Recall that, for a Euclidean mapping $f: \Omega \subset \R^d \map \R^d$,
the classical Piola identity reads as
$$
{\rm div}[\cof(\dd f)]=0,
$$
where $\cof(\dd f)$ is the cofactor matrix of the $(d\times d)$-matrix $\dd f$.
Kupferman--Maor--Shachar \cite[Theorems~1--2]{kms} generalized
it to the mappings between Riemannian manifolds (also see \cite{ks}).
Here we only collect some results in \cite{kms} related to our subsequent
development.

Consider a smooth mapping $f: (M,g) \map (\tM,\tg)$ between the Riemannian manifolds.
We first derive some identities for general manifolds $\tM$ and then specialize to $\tM=\snu$.
As in $\S \ref{subsec: vector bdls}$, we can view $\dd f \in \G(T^*M \otimes f^*T\tM)$;
equivalently, $\dd f \in \Omega^1(M,f^*T\tM)$, that is, $\dd f$ is a $f^*T\tM$-valued differential $1$-form over $M$.
By a standard application of multi-linear algebra (see \cite[\S 2.1]{kms}),
$\cof(\dd f) \in \Omega^1(M,f^*T\tM)$ may be defined intrinsically.
Furthermore, define the co-derivative:
	\begin{equation}
	\delta_{\na^{f^*T\tM}}\,:\,  \Omega^1(M,f^*T\tM) \map  \Omega^0(M,f^*T\tM)
	\end{equation}
as the formal $L^2$--adjoint operator of the differential between $f^*T\tM$-valued differential forms.
Then
\begin{equation}\label{piola}
\delta_{\na^{f^*T\tM}} \big(\cof(\dd f)\big) = 0,
\end{equation}
which is known as the {\em Riemannian Piola identity}.
In the geometric literature, it is more common to denote $\delta_\na$ by $\na^*$ or $\na^\dagger$.

For our purpose, we need to express the Piola identity \eqref{piola} in terms of the extrinsic geometry.
More precisely, assume that $\iota: (\tM,\tg) \map (\R^D, \euc)$ is an isometric embedding into the Euclidean space.
Such $\iota$ always exists for large enough $D$, due to Nash's embedding theorem in \cite{nash}.
Let $B$ be the second fundamental form of $\tM$ with respect to $\iota$, {\it i.e.}, $B: \G(T\tM) \times \G(T\tM) \map \G((T\tM)^\perp)$,
such that, for $u,v\in\G(T\tM)$ and $\eta \in \G((T\tM)^\perp)$,
\begin{equation}\label{A, two for N}
\euc(B(u,v),\eta) = \euc(\na^{\tM \times \R^D}_v \eta, \dd\iota \circ u);
\end{equation}
also see the equation below (2.8) in \cite{kms}.

We remark in passing on the following notations in \cite{kms}:
$\na^{M \times \R^D}$ and $\na^{\tM\times \R^D}$ should be understood as $(\iota\circ f)^*\na^{\R^D}$ and $\iota^*\na^{\R^D}$, respectively;
that is, they are the pullback affine connections, where $\na^{\R^D}$ denotes the Levi-Civita connection on $\R^D$ as before.
Indeed, it can be checked that \eqref{A, two for N} is equivalent to the definition in $\S \ref{subsec: geom prelim}$ above.
Correspondingly, in \cite{kms}, symbols $M\times\R^D$ and $\tM\times\R^D$ should be understood respectively
as the pullback bundles $(\iota\circ f)^* T\R^D \equiv f^*(\iota^*T\R^D)$ and $\iota^* T\R^D$.

With the above preparations, we can now state the weak formulation of the Riemannian Piola identity as in \cite[Theorem $2$]{kms}.
It is derived by splitting the tangent bundle $T\R^D$ into the tangential ($T\tM$) and normal ($(T\tM)^\perp$) directions,
and by applying the pullback operations under $\iota$ and $f$ suitably:
\begin{align}\label{piola identity, weak}
&\int_M \big\langle (f^* \dd\iota) \circ \cof(\dd f),\, \na^{M \times \R^D}\zeta\big\rangle_{g\otimes \euc}\,\dvg \nonumber\\
&= \int_M \big\langle {\rm tr}_g [ (f^*B)(\dd f, \cof(\dd f))],\,\zeta\big\rangle_\euc \,\dvg
\qquad \text{ for each $\zeta \in \G_0(M \times \R^D)$}.
\end{align}
We understand the terms in \eqref{piola identity, weak} as follows: On the left-hand side,
\begin{enumerate}
\item[(i)]
$f^*\dd\iota : f^*T\tM \map f^*(\iota^* T\R^D) \cong T(M \times \R^D)$,
\item[(ii)]
$\cof(\dd f) : TM \map f^*T\tM$,

\smallskip
\item[(iii)]
$\na^{M \times \R^D} \zeta : TM \map T(M \times \R^D)$;
\end{enumerate}
while, on the right-hand side,
\begin{enumerate}
\item[(i)]
$f^*B: f^*T\tM \times f^*T\tM \map f^*((T\tM)^\perp) \subset T(M \times \R^D)$,
\item[(ii)]
$f^*B(\dd f, \cof(\dd f)): TM \times TM \map  f^*((T\tM)^\perp)$,
\item[(iii)]
${\rm tr}_g[f^*B(\dd f, \cof(\dd f))] \in \G(f^*((T\tM)^\perp)) \subset \G[T(M \times \R^D)]$.
\end{enumerate}
Thus, both sides of \eqref{piola identity, weak} are well-defined so that we can integrate \eqref{piola identity, weak}
over $M$ with respect to the volume measure $\dvg$ induced by $g$.
By the Sobolev embeddings and an elementary approximation argument,
\eqref{piola identity, weak} is valid for any $f \in W^{1,p}(M;\tM)$ with $p \geq 2(d-1)$ ($p>2$ if $d=2$)
and $\zeta \in (W^{1,2}_0 \cap L^\infty)(M;\R^D)$.

\subsection{Proof of the Geometric Rigidity Theorem, Theorem \ref{thm: geom rigidity}}
Our proof follows the strategies in Friesecke--James--M\"{u}ller \cite{fjm}.
Nevertheless, it is far from a straightforward adaptation; we have to apply some ideas for dealing with harmonic
maps (see {\it e.g.}, H\'{e}lein \cite{helein} and H\'{e}lein--Wood \cite{hw}) to tackle
with the nonlinearities originated from the non-Euclidean geometry.\bigskip\\
\begin{proof}
We divide the proof into seven steps.

\smallskip
{\bf 1.} To start with, we derive an equation for $\cof(\dd f)-\dd f$.
It follows from the weak formulation of the Piola identity \eqref{piola identity, weak} that,
for each $\zeta \in \G_0((\iota\circ f)^*T \R^D)$,
\begin{align}\label{new, piola 1}
&\int_M \big\langle (f^* \dd\iota) \circ [\cof(\dd f)-\dd f],\, \na^{(\iota\circ f)^*T \R^D}\zeta\big\rangle_{g\otimes \euc}\,\dvg \nonumber\\
&\,= \int_M \big\langle {\rm tr}_g [(f^*B)(\dd f, \cof(\dd f))\big],\,\zeta\big\rangle_\euc \,\dvg\nonumber\\
&\,\quad - \int_M \big\langle (f^*\dd\iota) \circ \dd f,\, \na^{(\iota\circ f)^*T \R^D}\zeta\big\rangle_{g\otimes \euc}\,\dvg.
\end{align}
As before, $\na^M$ and $\na^{\R^D}$ denote  the Levi-Civita connections on $M$ and $\R^D$, respectively.
In view of the definitions of the co-derivative and $\langle\bullet,\bullet\rangle_{g\otimes\euc}$ and the fact that
$$
(f^*\dd\iota)\circ \dd f=\dd (\iota\circ f) \in \G_0((\iota\circ f)^*T\R^D),
$$
we have
\begin{align*}
&\int_M \big\langle (f^*\dd\iota) \circ \dd f,\, \na^{(\iota\circ f)^*T \R^D}\zeta\big\rangle_{g\otimes \euc}\,\dvg\nonumber\\
&= \int_M \big\langle  \dd(\iota\circ f),\,\na^{(\iota\circ f)^*T\R^D}\zeta\,\big\rangle_{g\otimes \euc}\,\dvg \nonumber\\
&= \int_M \big\langle \delta_{\na^{(\iota\circ f)^*T\R^D}} \circ  {\na^{(\iota\circ f)^*T\R^D}} \,(\iota\circ f),\,\zeta\big\rangle_{g\otimes \euc}\,\dvg\nonumber\\
& = \int_M {\rm tr}_g\big\{\big\langle \delta_{\na^{(\iota\circ f)^*T\R^D}} \circ  {\na^{(\iota\circ f)^*T\R^D}} \,(\iota\circ f),\,\zeta\big\rangle_{\euc}\big\}\,\dvg.
\end{align*}
Recall that, for a vector bundle $E$ over $M$, the Bochner Laplacian
$$
{\rm tr}_g(\delta_{\na^E}\circ\na^E) \qquad \mbox{over $\G(E)\equiv\Omega^0(M,E)$}
$$
coincides with the Hodge Laplacian, {\it i.e.},
the negative of the Laplace--Beltrami operator $\Delta_g$.
Then, identifying $(\iota\circ f)^* T\R^D$ with the trivial bundle $M\times\R^D$
and hence viewing both $\iota\circ f$ and $\zeta$ as mappings from $M$ to $\R^D$,
we may infer from \eqref{new, piola 1} that
\begin{align}\label{new, piola 3}
&\int_M \big\langle (f^* \dd\iota) \circ [\cof(\dd f)-\dd f],\, \na^{M \times \R^D}\zeta\big\rangle_{g\otimes \euc}\,\dvg \nonumber\\
& = \int_M \big\langle {\rm tr}_g [(f^*B)(\dd f, \cof(\dd f))],\,\zeta\big\rangle_\euc \,\dvg
    +\int_M \big\langle \Delta_g (\iota \circ f), \,  \zeta\big\rangle_{\euc}\,\dvg.
\end{align}
Now, using the identity
$$
\delta_{\na^{M \times \R^D}} \circ (f^*\dd\iota) = \delta_{\na^{f^*T\tM}}
$$
and the arbitrariness of the test differential form $\zeta \in \G_0((\iota\circ f)^*T \R^D)$, we have
\begin{align}\label{cof df - df}
&\delta_{\na^{f^*T\tM}}[\cof\,(\dd f)-\dd f] \nonumber\\
&= \Delta_g (\iota \circ f) + {\rm tr}_g[(f^*B) (\dd f,\dd f)]
   + {\rm tr}_g[(f^*B)(\dd f, \cof\,(\dd f)-\dd f)].
\end{align}
We will carry out the estimates for $f$ based on this equation.

\smallskip
{\bf 2.} We decompose $f=w+z$ with
\begin{equation}\label{z equation}
\begin{cases}
\Delta_g(\iota \circ z) = \delta_{\na^{f^*T\tM}} [\cof\,(\dd f)-\dd f]\\
\qquad\qquad\quad\, -{\rm tr}_g[(f^*B) (\dd f, \cof\,(\dd f)-\dd f)]\qquad \text{ on $M$},\\
z|_{\p M}=0
\end{cases}
\end{equation}
and
\begin{equation}\label{w equation}
\begin{cases}
\Delta_g(\iota\circ w) = -{\rm tr}_g[(f^*B)(\dd f,\dd f)]\qquad \text{ on } M,\\
w|_{\p M}=0.
\end{cases}
\end{equation}
This is possible, since $f|_{\p M}=0$ by assumption. We set
$$
\epsilon := \|\dist(\dd f, \so(g, \tg))\|_{L^2(M,g)}
$$
and assume $\epsilon \leq 1$ without loss of generality.
Since $\dd f \in \so(g, \tg)$ implies that ${\rm cof}\,(\dd f)=\dd f$ {\it a.e.} ({\it cf.} \cite[Corollary 4]{kms}),
there is a uniform constant $C=C(M,g)$ such that
\begin{equation*}
|\cof\,(\dd f)-\dd f|^2 \leq C\,\dist^2(\dd f, \so(g,\tg))\qquad \text{\it a.e.}.
\end{equation*}
It follows that
\begin{equation}
\|\cof\,(\dd f)-\dd f\|_{L^2(M,g;\tM,\tg)} \leq C\,\epsilon.
\end{equation}
Notice that, if $B \equiv 0$, {\it i.e.}, $\tM$ is Euclidean, then $w$ is taken to be a harmonic function.
This agrees with the case in Friesecke--James--M\"{u}ller \cite{fjm}.

\smallskip
{\bf 3.} We first derive the estimate for $z$.
Multiplying $\iota \circ z$ to \eqref{z equation} and recalling that $\iota$ is an isometric embedding,
we obtain
\begin{align}\label{new, dz}
\int_M |\dd z|^2\dvg
\leq \, & \int_M \big|\langle \na^{f^*T{\widetilde{M}}}(\iota \circ z), \, \cof(\dd f)-\dd f\rangle\big|\dvg\nonumber\\
   \,& + \int_M \big| z\,{\rm tr}_g[ (f^*B) (\dd f, \cof\,(\dd f)-\dd f)] \big|\dvg,
\end{align}
where the norm of $\dd z$ is taken with respect to both metrics $g$ and $\tg$:
\begin{equation}\label{new_modulos}
|\dd z| := \sqrt{\langle \dd z, \dd z\rangle_{g\otimes \tg}}.
\end{equation}

From now on, we focus on the case: $(\tM,\tg) = (\snu,\can)$. In this case, we have
\begin{equation}\label{new_harmonic map on sphere}
(f^*B)(\dd f,{\rm cof}(\dd f)-\dd f) = f\langle \dd f,{\rm cof}(\dd f)-\dd f\rangle_{g\otimes\can};
\end{equation}
see \cite[(26)]{hw}. Also, $|f|=1$. Thus, for some $C=C(M,g,\|f\|_{W^{1,\infty}(M,g)})$,
\begin{equation*}
\int_M |\dd z|^2\dvg
\leq C \Big\{ \int_M |\dd z| \big|\cof\,(\dd f)-\dd f\big|\dvg
+ \int_M|z|\big|\cof\,(\dd f)-\dd f\big|\dvg\Big\}.
\end{equation*}
{Here and hereafter, for notational convenience, we write
\begin{align*}
\|f\|_{W^{1,\infty}(M,g)} \equiv \|f\|_{W^{1,\infty}(M,g;\snu,\can)} \equiv  \|f\|_{W^{1,\infty}(M,g;\R^{d+1},\mathfrak{e})}.
\end{align*}
The latter equality holds since $(\snu,\can)$ is isometrically embedded in $(\R^{d+1},\mathfrak{e})$.}
Since $M$ is bounded and $z|_{\p M}=0$, we deduce from the Cauchy--Schwarz inequality,
$|z|=|\iota\circ z|_{g\otimes\euc}$,
and the Poincar\'{e} inequality (for functions on manifold $M$, {\it i.e.}, $0$-forms):
\begin{equation*}
\int_M |\phi|^2 \dvg \leq c_0^2 \int_M |\dd\phi|^2\dvg\qquad \text{ for each $\phi \in W^{1,2}_0(M;\R^D)$}
\end{equation*}
with $c_0=c_0(M,g)$
that
\begin{equation}\label{estimate for z}
\int_M |\dd z|^2\dvg \leq C(M,g, \|f\|_{W^{1,\infty}(M,g)})\,{\epsilon}^2.
\end{equation}

\smallskip
{\bf 4.} Next, we bound $w$ from \eqref{w equation}. Using $f=w+z$, $f \in \tM=\snu$ \emph{a.e.} on $M$,
and the specific geometric properties of $\snu$
{(namely identity~\eqref{new_harmonic map on sphere})}, we have
\begin{align*}
\int_M\big(|\dd w|^2+|w|^2|\dd f|^2\big)\,\dvg = -\int_{M}wz|\dd f|^2\,\dvg.
\end{align*}
Since $\|f\|_{W^{1,\infty}(M,g)} \leq C_1$, using the Poincar\'{e} inequality, we can estimate
\begin{align*}
\int_M \big(|\dd w|^2+|w|^2|\dd f|^2\big)\,\dvg &\leq C_1^2 \int_M|w||z|\,\dvg\\
&\leq \frac{1}{2}\int_M |\dd w|^2\,\dvg + 8C_1^4 c_0^4 \int_M |\dd z|^2\,\dvg.
\end{align*}
Thus, together with estimate~\eqref{new, dz} for
${\dd z}$, we conclude
\begin{align}\label{estimate for w}
\int_M |\dd w|^2\,\dvg \leq  C\big(M,g, \|f\|_{W^{1,\infty}(M,g)}\big) \,\epsilon^2.
\end{align}

In view of the Poincar\'{e} inequality again, \eqref{estimate for z}--\eqref{estimate for w}
can be summarized as
\begin{equation}\label{H1 estimate for w,z}
\|(w,z)\|_{W^{1,2}(M,g)}
\leq C(M,g)\,\epsilon.
\end{equation}
The $W^{2,2}$-estimates for $w$ can be derived directly from \eqref{w equation}:
\begin{align}\label{new, W2,2 estimate for w}
\|w\|_{W^{2,2}(M,g;\,\tM,\tg)}
&\leq C_1 \big(\|\Delta_g(\iota\circ w)\|_{L^2(M,g)} + \|w\|_{L^2(M,g)}\big)\nonumber\\
&\leq C_1\big(\big\|f|\dd f|^2\big\|_{L^2(M,g)} + \|w\|_{L^2(M,g)}\big)\nonumber\\
&\leq C(M,g,\|f\|_{W^{1,\infty}(M,g;\,\tM,\tg)})\, \epsilon.
\end{align}
In the first inequality above, we have used the Calder\'{o}n--Zygmund estimates on $M$ (see, {\it e.g.} Wang \cite{wang})
so that $C_1$ depends only on the
{$C^3$}--geometry of $(M,g)$.
The second inequality follows from \eqref{w equation}. The final inequality holds
by the assumption that $\|f\|_{W^{1,\infty}(M,g)}\leq C$ and \eqref{H1 estimate for w,z}.

\smallskip
{\bf 5.}
To proceed, {let us} estimate up to the $W^{4,2}$-norm of $w$. We {\em claim} that
\begin{equation}\label{W4,2 for w}
\|w\|_{W^{4,2}(M,g)} \leq C(M,g, \|f\|_{W^{2,\infty}\cap W^{3,2}(M,g)})\,\epsilon.
\end{equation}
To see this, taking two derivatives to the right-hand side of \eqref{w equation}
and then expressing it in local coordinates for the sake of clarity, we have
\begin{equation*}
\int_M|\na^2 \Delta_g w|^2\,\dvg = \int_{M} \Big|g^{k\ell}\frac{\p}{\p x^k} \frac{\p}{\p x^\ell} \Big(f g^{pq}\can_{ij} \frac{\p f^i}{\p x^p}\frac{\p f^j}{\p x^q}\Big)\Big|^2\,\dvg.
\end{equation*}
By the chain rule, we obtain
\begin{align}\label{new_C4 est}
&\int_M|\na^2 \Delta_g w|^2\,\dvg\nonumber\\
&\leq C\int_M \Big(|\na^2f|^2|\na f|^4 + |f|^2|\na^2f|^4 + |f|^2|\na^3 f|^2|\na f|^2\Big)\,\dvg,
\end{align}
where $C$ depends only on the
{$C^3$}-geometry of $M$. A similar computation yields
\begin{align}\label{new_C3 est}
\int_M|\na \Delta_g w|^2\,\dvg \leq C\int_M \big(|\na f|^3 + |f||\na f||\na^2 f|\big)^2\,\dvg.
\end{align}
By the $W^{1,2}$-estimate~\eqref{H1 estimate for w,z} for $f$ and the assumption
that $\|f\|_{W^{2,\infty}\cap W^{3,2}(M,g)}<\infty$, the right-hand sides of \eqref{new_C4 est}--\eqref{new_C3 est}
are both controlled by $\epsilon^2$.
Therefore, we may conclude the \emph{claim} from the Calder\'{o}n--Zygmund estimates on $M$.

\smallskip
{\bf 6.}
{We can now conclude the proof for  $d\in \{2,3\}$. Indeed,} utilizing the Sobolev embedding: $W^{4,2} \emb W^{2,\infty}$,
it follows from
{\eqref{new_C4 est}--\eqref{new_C3 est} and \eqref{new, W2,2 estimate for w}} that
$$
\|w\|_{W^{2,\infty}(M,g)} \leq C\,\epsilon,
$$
where $C=C (M,g, \|f\|_{W^{2,\infty}\cap W^{3,2}(M,g)})$.
Thus, by the fundamental theorem of calculus, there is a
{$d \times d$ matrix $\RR$} such that
\begin{align}\label{osc}
\sup_{x \in M} |\dd w(x) - \RR | \leq C \epsilon,
\end{align}
{where the norm $|\bullet|$ is defined as in \eqref{new_modulos} above, with $\dd w-\RR$ in lieu of $\dd z$.}
On the other hand, by the definition of $\epsilon$ and \eqref{estimate for z}, we have
\begin{align*}
&\big\|\dist(\dd w, \so(g,\can))\big\|_{L^2(M,g)}\\
&\leq
\big\|\dist(\dd f, \so(g,\can))\big\|_{L^2(M,g)}
+ \big\|\dist(\dd z, \so(g,\can))\big\|_{L^2(M,g)}\leq C\epsilon.
\end{align*}
{Furthermore, the left-hand side of \eqref{osc} is invariant under the actions by isometries of $\R^{d+1}$ so}
that we can take $\RR \in \so(g,\can)$ in \eqref{osc};
that is, $\RR$ is a rigid motion.
Thanks to \eqref{estimate for z} and \eqref{osc}, the proof is now complete
{for  $d\in \{2,3\}$}.

\smallskip
{
{\bf 7.} Finally, we now explain how to modify the arguments in Steps~5--6 above to deal
with the general case: $d=\dim M$.
Observe first that Step~6 remains valid, once we can estimate $w$ in some norm that is stronger
than $\|\bullet\|_{W^{2,\infty}(M,g)}$.
It suffices to obtain a bound of form $\int_M|\na^{\mathbf{n}(d)}\Delta_g w|^2\,\dvg < \infty$, which is
equivalent to estimating $\|w\|_{W^{\mathbf{n}(d)+2,2}(M,g)}$ by the Calder\'{o}n--Zygmund estimates,
since the following Sobolev continuous embedding holds when $\mathbf{n}(d)>\frac{d}{2}$:
\begin{equation*}
W^{\mathbf{n}(d)+2,2}(M,g) \emb W^{2,\infty}(M,g).
\end{equation*}
Now we are left to bound $\int_M|\na^{\mathbf{n}(d)}\Delta_g w|^2\,\dvg$.
Notice that
\begin{align*}
\int_M|\na^{\mathbf{n}(d)}\Delta_g w|^2\,\dvg
= \int_M \Big|\na^{\mathbf{n}(d)}\Big(f g^{pq}{\rm can}_{ij} \frac{\p f^i}{\p x^p}\frac{\p f^j}{\p x^q} \Big)\Big|^2\,\dvg.
\end{align*}
The integrand on the right-hand side contains, by the Leibniz rule, at most $\mathbf{n}(d)+1$ derivatives in both $f$ and the metric.
Since $(M,g)$ has finite volume as a submanifold of $(\snu,\can)$, the right-hand side is controlled
by $\|f\|_{W^{\mathbf{n}(d)+1,\infty}(M,g)}$ and the $C^{\mathbf{n}(d)+2}$-geometry of $(M,g)$. This completes the proof.
}
\end{proof}

\subsection{Remarks}
Concerning the geometric rigidity theorem (Theorem~\ref{thm: geom rigidity}) and its proof,
we have the following remarks in order:

\smallskip
{\bf 1.} Starting from Step~3 in the above proof of Theorem~\ref{thm: geom rigidity},
our argument deviates from Step~2 in the proof of  \cite[Proposition~3.4]{fjm}.
In \cite{fjm}, map $w$ is harmonic, so an interior bound for $\|\na^2w\|_{L^2}$ follows easily; see (3.16) therein.
However, in our case, $ \Delta_g(\iota\circ w) = -{\rm tr}_g[(f^*B)(\dd f,\dd f)]$ in $M$; see \eqref{w equation}.
We proceed as in \eqref{new, W2,2 estimate for w} by the estimate:
$\|w\|_{W^{2,2}} \lesssim \|\Delta_g(\iota\circ w)\|_{L^2(M,g)} + \|w\|_{L^2(M,g)}$.
This is where the additional regularity assumptions on $f$ come into play.

\smallskip
{\bf 2.} We have imposed an extra condition $f=0$ (hence $w=0$) on $\p M$ in Theorem~\ref{thm: geom rigidity},
while no boundary condition for $f$ is needed in Proposition~\ref{propn: geometric rigidity}.
This is due to the fact that $w$ is not a \emph{harmonic function} in our case.
Instead, PDE~\eqref{w equation} for $w$ is a ``perturbed'' harmonic map equation:
when $w=f$, it is exactly the harmonic map equation.
Therefore, the weighted Hessian estimate as in \cite[(3.26)]{fjm} is no longer valid in general,
thus preventing us from adapting the arguments in \cite{fjm}
to derive the estimates up to the boundary,
unless $f=0$ on $\p M$.

\smallskip
{\bf 3.} It would be interesting to further generalize Theorem~\ref{thm: geom rigidity}
to general ambient manifolds $(\tM,\tg)$. In this case, it is natural to define that $\RR \in \so(g,\tg)$ is a rigid motion
if and only if $\RR$ is an isometry of the tangent bundle of the ambient manifold with respect to $\tg$.
We may view $\RR \in \so(g,\tg)$ by identifying $\RR$ with its restriction on $TM$.
Such a definition entails that,
for most of the ambient manifolds $(\tilde{M}, \tilde{g})$,
$\so(g,\tg)=\{{\rm Id}_M\}$,
whence \eqref{3.1b} trivially holds with $C=1$.
The nontrivial case that $\so(g,\tg)$ has more than one element occurs when, roughly speaking,
$(T\tM,\tg)$ has a large degree of symmetries, \emph{e.g.},
when $(\tM,\tg)$ is a space form (\emph{i.e.}, a Riemannian manifold with constant sectional curvature).

\smallskip
{\bf 4.} In the case of a general ambient manifold $(\tM,\tg)$,
the main difficulty lies in obtaining a smallness estimate for $\|\dd w\|_{L^2(M,g;\tM,\tg)}$.
Indeed, when multiplying $w$ to both sides of Eq.~\eqref{w equation} and integrating over $M$,
we obtain
\begin{align*}
\int_M |\dd w|^2\,\dvg
= -\int_M\Big\{g_{\gamma\delta}g^{ij}w^\delta\big(\tG^\gamma_{\alpha\beta}\circ f\big) (\dd f)^\alpha_i (\dd f)^\beta_j
\Big\}\,\dvg
\end{align*}
in local coordinates with Einstein's summation convention,
where $\tG^\gamma_{\alpha\beta} \in C^2(\tM,\tg)$ are the Christoffel symbols on $\tM$ of the Levi-Civita connection for $\tg$; see \cite{hw,es}.
There is no apparent structure for the integrand on the right-hand side
that leads to $\|\dd w\|_{L^2(M,g;\tM,\tg)} \lesssim \epsilon$ as in \eqref{estimate for w}.

\smallskip
{\bf 5.} In the recent preprint {\rm \cite{new-akm}}, Alpern--Kupferman--Maor further extended their asymptotic rigidity theorem
in {\rm \cite{kms}} to hypersurfaces in space forms.
As remarked in {\rm \cite{new-akm}}, a quantitative result in the form of a geometric rigidity theorem may help to extend
the asymptotic rigidity theorem therein to arbitrary ambient Riemannian manifolds.
It would be interesting to further investigate such extensions.

\section{Asymptotic Rigidity of Elastic Membranes}
\label{sec: convergence}

In this section, we formulate and prove an asymptotic rigidity theorem for elastic membranes,
{\it i.e.}, immersed hypersurfaces $M^\e$ in the Euclidean space $(\R^{d+1}, \euc)$, in which
$\euc$ is the Euclidean metric as before.
The theorem addresses the convergence of both deformations and extrinsic geometries.

\begin{theorem}\label{thm: convergence}
Let $\mi$ be a sequence of $d$-dimensional Riemannian manifolds and $p>d \geq 2$.
Let $\Phi^\varepsilon:(\mi,\gi) \emb (\R^{d+1}, \euc)$ be $W^{2,p}_\loc$-isometric immersions.
Let $M$ be a Riemannian manifold  with $W^{1,p}_\loc$-metric $g$ so that there are
bi-Lipschitz homeomorphisms $\ffi:M \map \mi$ whose bi-Lipschitz constants are uniformly bounded on compact sets such that
\begin{equation}
\ffiu [\gi] - g \longrightarrow 0\qquad \text{ in $W^{1,p'}_\loc(M)$ $\,$ for $p'=\frac{p}{p-1} \in [1,2)$},
\end{equation}
and that $\Phi^\varepsilon\circ \ffi$ are uniformly bounded in $W^{2,p}_{\rm loc}(M,g;\,\R^{d+1}, \euc)$.
Then, after passing to a subsequence if necessary, $\Phi^\varepsilon\circ\ffi$ converges weakly in $W^{2,p}_\loc$ to an isometric
immersion $\bigphi:(M,g) \map (\R^{d+1}, \euc)$ so that its second fundamental form
is a weak $L^p_\loc$-limit of the second fundamental forms of $\Phi^\varepsilon$, obeying the Gauss--Codazzi equations \eqref{gauss}--\eqref{codazzi}.
\end{theorem}
\begin{proof}  Throughout the proof, unless otherwise specified,
all the Sobolev spaces ${\bf X}=W^{k,p}, W^{-k,p}, L^p,\ldots$ are understood as ${\bf X}(M,g;\R^{d+1},\euc)$.
The proof is divided into six steps.

\smallskip
{\bf 1.} We first fix some notations.
Denote by $\na$ the Levi-Civita connection on $(M,g)$ and by $\nai$ the Levi-Civita connection on $(\mi,\gi)$. By assumption, $\Phi^\varepsilon: (\mi,\gi)\to(\R^{d+1},\euc)$ are isometric immersions and $\ffi$ are Lipschitz homeomorphisms. Then
\begin{equation*}
\comp: \big(M,  (\ffi)^*\gi\big) \longrightarrow \eucl
\end{equation*}
are also isometric immersions.
We write $\nnue$ for the outward unit normal vector field of $\comp$; in the proof below,
we view it as defined on $\R^{d+1}$.

We also set
	\begin{equation}\label{hat-g, hat-nabla}
	\hatgi := \ffiu \gi, \quad \hatnai := \ffiu \nai  \qquad\mbox{on $M$},
	\end{equation}
and
\begin{equation}\label{hat-b}
\hatbi(X,Y):=-\euc(\barna_{\push(X)} \nnue, \push(Y))\quad \text{for } X,Y \in \G(TM),
\end{equation}
which are well-defined since $\ffi$ are bi-Lipschitz homeomorphisms. For notational convenience,  write
\begin{equation*}
\xe := (\ffi)_*X = \dd\ffi(X) \in \G(TM^\e)  \qquad\,\mbox{for each $X \in \G(TM)$}.
\end{equation*}
That is, $X$ and $\xe$ are {\em $\ffi$-related}. The same convention applies to $Y_\varepsilon, Z_\varepsilon, W_\varepsilon\ldots$.
Also, by the locality of the theorem, without loss of generality,
we may take $M$ to be compact, so that the subscripts ``$_{\loc}$'' are dropped from now on.

The tensor $\hatbi: \G(TM) \times \G(TM) \to \R$ defined in \eqref{hat-b} is the second fundamental form
of the isometric immersion $\comp: (M,\hatgi)\to\eucl$.
Note that our convention here for the second fundamental form is slightly different from that in \S 2:
we view $\hatbi$ as a $\R$-valued function, which is more convenient for the case of codimension one.

The definition in \eqref{hat-b} is motivated by the following observations:
\begin{itemize}
\item
If $\comp$ is smooth for each $\varepsilon$, then, by definition (see \S\ref{sec: geometry}),
\begin{equation}\label{new_def_hat B}
\hatbi(X,Y) = \euc( \barna_{\push(X)} {\push(Y)},\nnue).
\end{equation}
The right-hand side of \eqref{new_def_hat B} can be understood as follows: we can locally extend $\push X$ and $\push Y$,
which are vector fields on $\comp(M)\subset \R^{d+1}$, to vector fields $X'_\varepsilon$ and $Y'_\varepsilon$ on $\R^{d+1}$, respectively,
and then set $\hatbi(X,Y) = \euc\big( \barna_{X'_\varepsilon} {Y'_\varepsilon},\nnue\big)$.
This is independent of the choice of extensions $X'_\varepsilon$ and $Y'_\varepsilon$; see do Carmo \cite[pp.126--127, \S 6]{doc}.

\smallskip
\item
Since $\comp$ is uniformly bounded in $W^{2,p}$, the right-hand side of \eqref{new_def_hat B}
is a product of two $W^{1,p}$-terms and one $L^p$-term. Since $p>d=\dim M$,
we see that $\hatbi$ is uniformly bounded in $L^p$ (see Step~5 below for details).

\smallskip
\item
The above remarks justify the computation below:
\begin{align*}
\quad \hatbi(X,Y)=&\, \push (X) \{\euc (\push(Y),\nnue)\} \\
&\,- \euc(\barna_{\push(X)} \nnue, \push(Y)).
\end{align*}
Since $\push(Y)$ is tangential and $\nnue$ is vertical, both with respect to the immersed hypersurface $\comp(M)$,
the first term on  the right-hand side vanishes. Therefore, we arrive at \eqref{new_def_hat B}.
\end{itemize}

In addition, by passing to subsequences if necessary, we have
\begin{align*}
\phii \circ \ffi \,\,\weak\,\, \bigphi \qquad \text{in } W^{2,p}.
\end{align*}
Note that $D^\top\bigphi \cdot D\bigphi \in W^{1,p}$ for $p>d$, thanks to the Sobolev--Morrey embedding.
Since $\hatgi:=(\ffi)^*\gi = (\comp)^*\euc$ converges in $W^{1,p'}$ to $g$ on $M$,
by a compactness argument and the uniqueness of limits, we conclude that $g = D^\top\bigphi \cdot D\bigphi$ as $W^{1,p}$-tensor fields.
Thus, $\bigphi$ is an isometry. Moreover, $|\det (D\bigphi)| = \sqrt{\det g} >0$ in the {\it a.e.} sense, so $\bigphi$ is also an immersion.

\smallskip
{\bf 2.} In order to prove the weak convergence of second fundamental forms,
we will prove that $\hatbi$ is an \emph{approximate solution} for the Gauss--Codazzi equation on $(M,g)$,
which will be made precise in Lemma~\ref{lemma: chenli} below.
In what follows, we let $X,Y,Z,W \in \G(TM)$ be arbitrary, and consider $\ffi$-related vector fields $\xe,\ye,\ze,\we \in \G(T\mi)$.
We identify them (without relabelling) with extensions $X',Y',Z',W'$, $X'_\e,Y'_\e, Z'_\e, W'_\e \in \G(T\R^{d+1})$ as in Step~1.

As before, $\comp: (M, \hatgi) \longrightarrow \eucl$ are isometric immersions for $\hatgi\equiv (\ffi)^*\gi$.
Then $\hatbi$ defined in \eqref{hat-b} satisfies the Gauss equation:
\begin{equation*}
\hatbi(X,Z)\hatbi(Y,W) - \hatbi(X,W)\hatbi(Y,Z) = \curve(X,Y,Z,W),
\end{equation*}
where $\curve$ is the Riemann curvature tensor of $(M,\hatgi,\hatnai)$
and we have used the fact that $\euc(\nnue,\nnue)=1$.

Denoting by $R^\varepsilon$ the Riemann curvature tensor of $(\mi,\gi,\nai)$, we have
\begin{equation*}
\curve = (\ffi)^*R^\varepsilon.
\end{equation*}
This follows from the tensorial property of the Riemann curvature:  For each $P \in M$,
\begin{align*}
&[\ffiu R^\varepsilon](X,Y,Z,W)|_P\\
& = R^\varepsilon(\xe,\ye,\ze,\we)|_{\ffi(P)}\\
& =\gi( \nai_\xe\nai_\ye\ze - \nai_\ye\nai_\xe \ze-\nai_{[\xe,\ye]}\ze ,\, \we) \big|_{\ffi(P)}\\
&=[(\ffi)^*\gi](\ffi^*\{\nai_\xe\nai_\ye\ze - \nai_\ye\nai_\xe \ze-\nai_{[\xe,\ye]}\ze\},\,(\ffi)^*\we)\big|_P\\
&= \hatgi( \hatnai_{X} \{\ffiu(\nai_\ye\ze)\} - \hatnai_{Y}\{\ffiu(\nai_\xe\ze)\}
- \hatnai_{\ffi^*[\xe,\ye]}(\ffi^*\ze),\, W )\big|_P\\
&= \hatgi(\hatnai_{X} \hatnai_{Y}Z- \hatnai_{Y}\hatnai_{X}Z
   - \hatnai_{[X, Y]}Z,\,W)\big|_P\\
&=: \hatri (X, Y, Z, W)|_P.
\end{align*}
For the penultimate equality, we have used
the Lie bracket identity $f_*([X,Y])=[f_*X,f_*Y]$.

From the computations above, we infer that
\begin{align}\label{pullback gauss eq}
\hatbi(X,Z)\hatbi(Y,W) - \hatbi(X,W)\hatbi(Y,Z) = R(X,Y,Z,W) + [{\rm Error}]_1,
\end{align}
where $R$ denotes the  Riemann curvature tensor of $(M,g,\na)$, and
\begin{equation*}
[{\rm Error}]_1:=\{\curve - R\}(X,Y,Z,W).
\end{equation*}

On the other hand, the Codazzi equation for $\hatbi$ reads as
\begin{equation}\label{pullback codazzi eq}
\barna_X\big(\hatbi(Y,Z)\nnue\big)-\barna_Y\big(\hatbi(X,Z)\nnue\big)=0.
\end{equation}

\smallskip
{\bf 3.} To proceed, we invoke the weak continuity of the Gauss--Codazzi equations  \eqref{gauss}--\eqref{codazzi}  established
in \cite{chenli}.
We employ a variant of Proposition \ref{propn: chenli, weak continuity thm for GCR eq},
which deals with the weak continuity of ``approximate solutions''
for the Gauss--Codazzi equations  \eqref{gauss}--\eqref{codazzi} (see \cite[Remark $4.1$]{chenli}).

\begin{lemma}\label{lemma: chenli}
Let $(M,g)$ be a $d$-dimensional Riemannian manifold with
{$W^{1,p}_\loc\cap L^\infty_\loc$}-metric for $p>2$.
Suppose that the tensor fields
$$
\two^\varepsilon: \G(T\mi) \times \G(T\mi) \map \G((T\mi)^\perp)
$$
both have a uniform bound
in $L^p_\loc$ and are ``approximate solutions'' of the Gauss--Codazzi equations \eqref{gauss}--\eqref{codazzi} in the following sense{\rm :}
\begin{eqnarray}
&&\euc(\two^\varepsilon(u,w), \two^\varepsilon(v,z))
  - \euc( \two^\varepsilon(u,z), \two^\varepsilon(v,w))-R(u,v,w,z)
   = \mathcal{O}^{(1)}_\varepsilon,\label{gauss approx}\qquad\,\,\,\, \\
&&\barna_{v} \two^\varepsilon(u,w) - \barna_{u} \two^\varepsilon(v,w) = \mathcal{O}^{(2)}_\varepsilon,\label{codazzi approx}
\end{eqnarray}
for arbitrary fixed $u, v, w, z \in \G(TM)$,
where
$\mathcal{O}^{(1)}_\varepsilon, \mathcal{O}^{(2)}_\varepsilon \to 0$ in $W^{-1,r}_\loc$ as $\varepsilon \to 0$
{for some $r>1$},
and $R$ is the Riemann curvature tensor of $g$ and the Levi-Civita connection $\na$ on $M$.
Then $\{\two^\varepsilon\}$ converges weakly in $L^p_\loc$ to a weak solution of the Gauss--Codazzi equations \eqref{gauss}--\eqref{codazzi}.
\end{lemma}

We will apply this lemma to
{$\two^\varepsilon := \hatbi\nnue$}, with $(u,v,w,z)$ replaced by $(X, Y, Z, W)$.
By \eqref{pullback gauss eq}  and \eqref{pullback codazzi eq}, we see that $[{\rm Error}]_1 = \mathcal{O}^{(1)}_\varepsilon$
and $\mathcal{O}^{(2)}_\varepsilon\equiv 0$ here.
It suffices to show that $\mathcal{O}^{(1)}_\varepsilon\to 0$ in $W^{-1,r}$ for some $r>1$ as $\varepsilon \to 0$
and that $\two^\varepsilon$ is uniformly bounded in $L^p$.

\smallskip
{\bf 4.} Using the definition of the Riemann curvature tensor, $\mathcal{O}^{(1)}_\varepsilon=\{\curve - R\}(X,Y,Z,W)$ can be expressed as
\begin{align}\label{O1'}
\mathcal{O}^{(1)}_\varepsilon =\,& (\hatgi-g)(\hatnai_{X} \hatnai_{Y} Z,\, W)
  + g( \{ \hatnai_{X} - \na_{X}\}\na_{Y}Z,\, W)\nonumber\\
&+ g( \na_{X}\{\hatnai_{Y} - \na_{Y}\}Z,\, W) - (\hatgi-g)( \hatnai_{Y} \hatnai_{X} Z,\, W)\nonumber\\
&- g(\{ \hatnai_{Y} - \na_{Y}\}\na_{X}Z,\, W) - g(\na_{Y}\{\hatnai_{X} - \na_{X}\}Z,\, W) \nonumber\\
&- (\hatgi-g)(\hatnai_{[X,Y]}Z,\, W) - g( \{\hatnai_{[X,Y]}-\na_{[X,Y]}\}Z,\, W) \nonumber\\
=:\,& \sum_{\ell=1}^8 {J}_\ell.
\end{align}

We first observe that
\begin{equation}\label{bddness of hat-Gamma}
\text{$\hgai$ $\, $ is uniformly bounded in $L^p$.}
\end{equation}
Indeed, since $\comp$ is uniformly bounded in $W^{2,p}$, it follows that $\hatgi$ is uniformly bounded in $W^{1,p}$.
This can be seen from explicit computations: in the local orthonormal coordinates $\{e_\alpha\}$ on $(M,g)$,
\begin{equation*}
\hatgi_{\alpha\beta} = D(\comp)^\alpha_\gamma D(\comp)^\gamma_\beta.
\end{equation*}
The Sobolev--Morrey embedding then implies the uniform boundedness of $\{\hatgi\}$ in $C^0$.
Thus, \eqref{bddness of hat-Gamma} follows from the continuity of the matrix inverse and the formula:
\begin{equation*}
{\hgai}^\alpha_{\beta\gamma}
=\frac{1}{2}(\gi)^{\alpha\delta}\big\{\p_\beta \gi_{\delta\gamma}
 +\p_\gamma \gi_{\delta \beta}-\p_\delta \gi_{\gamma\beta}\big\}.
\end{equation*}

In what follows, we justify the convergence $J_\ell \to 0$ term by term.

\smallskip
\noindent
\underline{$J_1$ and $J_4$.} A direct computation yields that
\begin{align*}
\hatnai_{X} \hatnai_{Y} Z
&= X^\delta \p_\delta\big(Y^\alpha\p_\alpha Z^\beta\big)\p_\beta
 + X^\delta Y^\alpha (\p_\alpha Z^\beta) \hgai^\gamma_{\delta\beta}\p_\gamma\nonumber\\
&\quad\, +X^\delta Y^\alpha Z^\beta \hgai^\gamma_{\alpha\beta} \hgai^\kappa_{\delta\gamma}\p_\kappa
+ X^\delta \p_\delta\big( Y^\alpha Z^\beta \hgai^\gamma_{\alpha\beta} \big)\p_\gamma,
\end{align*}
where the last term of the right-hand side is most singular.
We deduce from this equality that $\hatnai_{X} \hatnai_{Y} Z$ is uniformly bounded in $W^{-1,p}$
for any given $X,Y,Z \in \G(TM)$. Thanks to the assumption of the strong convergence $\hatgi \to g$ in $W^{1,p'}$,
we then obtain that
$$
J_1 \to 0 \qquad \mbox{in $W^{-1,r}$ as $\varepsilon \to 0$}
$$
for any $r <d'=\frac{d}{d-1}$. The argument for $J_4$ is analogous.

\smallskip
\noindent
\underline{$J_2$ and $J_5$.}$\,$  This is more direct. Since $\na_YZ \in L^p$,
$$
(\hatnai_X-\na_X)\na_YZ =X^i (\na_YZ)^j (\widehat{\G^\varepsilon}-\G)^k_{ij}\p_k \weak 0\qquad \mbox{in $L^{p/2}$}.
$$
Using $g \in W^{1,p}\emb C^0$, the Rellich lemma, and the Sobolev embedding,
$$
J_2 \to 0 \qquad\mbox{in $W^{-1,r}$ as $\varepsilon \to 0$}
$$
for all $r \in (1,\infty)$ if $d\le \frac{p}{2}$,
and for $r \in (1,\frac{dp}{2d-p})$ if $d>\frac{p}{2}$.
The argument for $J_5$ is analogous.

\smallskip
\noindent
\underline{$J_3$ and $J_6$.} $\,$ It suffices to argue for $J_3$.
As before, after passing to subsequences, $\hatnai_YZ-\na_YZ \weak 0$ in $L^p$
so that $\na_X(\hatnai_YZ-\na_YZ) \weak 0$ in $W^{-1,p}$.
Since $g \in W^{1,p}$, it follows from the Sobolev embedding that
$$
J_3 \to 0 \qquad\mbox{in $W^{-1,r}$ as $\varepsilon \to 0$ for any $r <d'=\frac{d}{d-1}$.}
$$

\smallskip
\noindent
\underline{$J_8$.} $\,$ Note that $[X,Y]=\na_XY-\na_YX$ as $\na$ is the Levi-Civita connection on $(M,g)$,
hence $[X,Y]\in L^p$.
It follows that
$$
\{\hatnai_{[X,Y]}-\na_{[X,Y]}\}Z = [X,Y]^i Z^j (\widehat{\G^\varepsilon}-\G)^k_{ij} \p_k \weak 0
\qquad\mbox{in $L^{p/2}$}.
$$
The rest of the argument is similar to that for $J_2$ and $J_5$ above.

\smallskip
\noindent
\underline{$J_7$.} Finally, note as before that $\hatnai_{[X,Y]}Z$ is uniformly bounded in $L^{p/2}$.
In $J_7$, this term is paired with $\hatgi-g$, which converges strongly to zero in $W^{1,p'}$.
By Sobolev embedding, for $p>d$, we have the compact embedding: $W^{1,p'} \emb L^{\frac{dp'}{d-p'}}=L^{\frac{dp}{d(p-1)-p}}$,
which is greater than or equal to the H\"{o}lder conjugate $\frac{p}{p-2}=(\frac{p}{2})'$.
Thus, $J_7 \to 0$ in $L^1$. By Sobolev embedding again,
$$
J_7 \to 0 \qquad \mbox{in $W^{-1,r}$ as $\varepsilon \to 0$ for any $r <d'=\frac{d}{d-1}$}.
$$

To summarize, by the arguments above, we conclude that $\mathcal{O}^{(1)}_\varepsilon \map 0$  in $W^{-1,r}$
for some $r>1$, where $\mathcal{O}^{(1)}_\varepsilon$ is defined in \eqref{O1'}.

\smallskip
{\bf 5.} Now we show that $\two^\varepsilon=\hatbi\nnue$ is uniformly bounded in $L^p$. Recall from \eqref{new_def_hat B} that
\begin{align*}
\two^\varepsilon(X,Y) = -\euc(\barna_{\push X}\nnue,\,\push Y)\nnue
\end{align*}
for any vector fields $X,Y \in \G(TM)$. For ease of notations, we write here $\xxe:=\push X$ and $\yye:=\push Y$.
By assumptions, they are uniformly bounded in $W^{1,p}$.
Thus, for  the Euclidean coordinate frame $\{\p_i\}$ on $\R^{d+1}$, we may express
\begin{align}\label{new_two}
\two^\varepsilon(X,Y) = - (\xxe)^i \big[\p_i(\nnue)^j\big](\yye)^j \nnue.
\end{align}

Assume for the moment that $\nnue$ is uniformly bounded in $L^p$.
Then the right-hand side of \eqref{new_two} is a product of three terms uniformly bounded in $W^{1,p}$
and another term uniformly bounded in $L^p$.
By assumption $p>d=\dim M$, the Sobolev--Morrey embedding shows that $W^{1,p} \emb C^0$.
Hence, $\two^\varepsilon$ is uniformly bounded in $L^p$.

To justify the above claim, we make use of the following expression for $\nnue$:
\begin{align}\label{new_normal}
\nnue = \frac{\dd (\comp)^1 \wedge \cdots \wedge \dd (\comp)^d}{\|\dd (\comp)^1 \wedge \cdots \wedge \dd (\comp)^d\|} \mres \dve,
\end{align}
where $\dve$ is the Euclidean volume form on $\eucl$, $\mres$ is the interior multiplication,
and $\|\bullet\|$ is the mass norm for vector fields.
This equality is understood modulo obvious isomorphisms between the tangent and cotangent bundles.
Note that the denominator $\|\dd (\comp)^1 \wedge \cdots \wedge \dd (\comp)^d\| \geq c_0$ for a uniform constant $c_0>0$,
since $\Phi^\e$ are isometric immersions (hence non-degenerate) and $\ffi$ have uniformly bounded bi-Lipschitz constants.
Moreover, $\dd (\comp)^1 \wedge \cdots \wedge \dd (\comp)^d$ is a wedge product of $d$ differential $1$-forms uniformly
bounded in $W^{1,p}$.
Again, since  $W^{1,p} \emb C^0$ for $p>d$, we deduce that $\nnue$ is uniformly bounded in $L^p$.

\smallskip
{\bf 6.} Finally, by Lemma \ref{lemma: chenli} and Steps~4--5, we conclude that $\hatbi\nnue$ converge weakly in $L^p$ to a weak solution of the Gauss--Codazzi equations \eqref{gauss}--\eqref{codazzi}.
This completes  the proof.
\end{proof}

\section{Continuous Dependence of the Deformation on the Cauchy--Green Tensor and Second Fundamental Form}
\label{sec: continuous dep}

In Ciarlet-Mardare \cite{cm3}, the following continuous dependence of the deformation was established:

\begin{theorem}[{\cite[Theorem $6.1$]{cm3}}]\label{thm: ciarlet}
Let $\Omega$ be a simply-connected, open, bounded subset of $\R^2$ with Lipschitz boundary $\p\Omega$.
Assume that $\Omega$ lies locally on the same side of $\p\Omega$, and $p>2$.
Define the spaces{\rm :}
$$
\mathbb{T}(\Omega) := \left\{(g,B)\, :\,
 \begin{array}{ll}
  g \in W^{1,p}(\Omega; {\rm Sym}^{+}_{2\times 2}T^*\Omega), B\in L^p(\Omega; {\rm Sym}_{2\times 2}T^*\Omega)\\
  \text{$(g,B)$ satisfies GCE \eqref{gauss}--\eqref{codazzi}}
 \end{array}
 \right\}
$$
and
\begin{equation*}
\mathbf{V}(\Omega) := W^{2,p}(\Omega; \R^3) / {\rm Isom}_+(\R^3),
\end{equation*}
which are equipped with the natural topologies inherited from the corresponding Sobolev spaces $W^{k,p}$ for $k\in\{0,1,2\}$.
Let $\Phi: \mathbb{T}(\Omega) \map \mathbf{V}(\Omega)$ map $(g,B)$ to the immersion $f$ such that $f$ is an isometric immersion
{from $(\Omega,g)$ to $(\R^3,\euc)$}
with the second fundamental form $B$. Then $\Phi$ is locally
{Lipschitz} continuous.
\end{theorem}

In Theorem \ref{thm: ciarlet} above, ${\rm Isom}_+(\R^3)$ denotes the group of orientation-preserving isometries in the $3$-D Euclidean
space, {\it i.e.}, ${\rm Isom}_+(\R^3) = \R^3 \rtimes \so(3)$,
and
{${\rm Sym}_{d\times d}$ is the space of symmetric $d\times d$ matrices
with superscript ``${}^+$'' to designate the positive definiteness.
}

In nonlinear elasticity, Theorem \ref{thm: ciarlet} states the continuous dependence of the deformation
of elastic bodies on the Cauchy--Green tensor (\textit{i.e.,} the metric) and the extrinsic geometry,
for $2$-D elastic bodies with lower regularity.
Its proof is based on a programme developed by P.~G. Ciarlet, C. Mardare, and S. Mardare; see \cite{cm1, cm2, cm3} and
the references cited therein.
The goal of the programme (as summarized in \cite{cgm}) is to extend the ``fundamental theorem of surface theory''
(that is, a simply-connected surface immersed in $\R^3$ can be uniquely recovered from its metric $g$ and second fundamental
form $B$ modulo ${\rm Isom}_+(\R^3)$--actions) to the case of surfaces with lower regularity (\textit{i.e.},
when $g \in W^{1,p}$ and $B \in L^p$ for $p>2$).
Also see Szopos \cite{s} for the higher dimensional case.

The proof of Theorem \ref{thm: ciarlet} in \cite{cm1, cm2, cm3} may be outlined as follows:
First, the Gauss--Codazzi equations  \eqref{gauss}--\eqref{codazzi}  are transformed into
two types of first-order, nonlinear, matrix-valued PDEs,
known as the {\em Pfaff} and {\em Poincar\'{e}} systems.
Then, applying the analytic results due to Mardare \cite{m03, m05, m07},
the Pfaff and Poincar\'{e} systems are solved,
and the continuous dependence of solutions in suitable Sobolev spaces is proved.
On the other hand, the transformation from the Gauss--Codazzi equations  \eqref{gauss}--\eqref{codazzi}
to the Pfaff--Poincar\'{e} systems in \cite{cm1, cm2, cm3} appears highly intricate,
which involves many different types of geometric quantities ({\it e.g.}, metrics, connections, curvatures, $\dots$)
in local coordinates as the entries of the same matrices of enormous size.

In \cite{chenli}, we provided a simpler, more direct approach of proving the existence
of $W^{2,p}$--isometric immersions with respect to the prescribed $W^{1,p}$--metrics
and $L^p$--second fundamental forms by using the Cartan formalism of exterior calculus on manifolds.
In what follows, we explain how Theorem \ref{thm: ciarlet} can be recovered by utilizing the method in  \cite{chenli}.

In fact, we establish an analogue of Theorem \ref{thm: ciarlet} on a simply-connected Riemannian manifold $(M,g)$ in arbitrary dimensions
and co-dimensions (but neglecting the effects of boundary).
For this, let $E$ be a $\R^k$-vector bundle over $d$-dimensional Riemannian manifold
$M$ with bundle metric $g^E$ and compatible bundle connection $\na^E$.
We use Latin letters $X,Y,Z,W,\ldots$ to denote the tangential vector fields in $\G(TM)$, and Greek letters $\xi,\eta\in \G(f(TM)^\perp)$ or $\G(E)$
to denote the vector fields on the normal bundle $f(TM)^\perp$ or on some given vector bundle $E$.
Also, $S_\eta$ is the shape operator determined by the second fundamental form $B$ via $S_\eta(X,Y) = \langle \eta, B(X,Y)\rangle$,
$\widetilde{\na}$ is the Levi-Civita connection on the ambient space $\R^{d+k}$,
$R$ denotes the Riemann curvature tensor on $M$,
and $R^E$ denotes the Riemann curvature tensor on $E$:
$$
R^E(X,Y) = [\na^E_X, \na^E_Y] - \na^E_{[X,Y]}.
$$
Then the Gauss, Codazzi, and Ricci equations on $E$ are as follows in order:
\begin{eqnarray}
&& \langle B(Y,W), B(X,Z)\rangle - \langle B(X,W), B(Y,Z)\rangle = R(X,Y,Z,W),\label{gauss equation}\\
&& \widetilde{\na}_Y B(X,Z) = \widetilde{\na}_X B(Y,Z),\label{codazzi equation}\\
&& \langle [ S_\eta, S_\xi ] X, Y\rangle = R^E(X,Y,\eta,\xi).\label{ricci equation}
\end{eqnarray}
In \cite[Theorem 5.2]{chenli}, we established the equivalence of the following three clauses for $p>d$:
\begin{enumerate}
\item[(i)] the existence of a weak solution $(g,B,\na^E) \in W^{1,p}_{\rm loc} \times L^p_{\rm loc} \times L^p_{\rm loc}$
of the Gauss--Codazzi--Ricci equations \eqref{gauss equation}--\eqref{ricci equation} as above;

\item[(ii)] the Cartan formalism (see \cite[\S 5.2]{chenli} for details);

\item[(iii)] the existence of a $W^{2,p}_{\rm loc}$--isometric immersion whose metric, second fundamental form,
and normal connection are the weak solution $(g,B,\na^E)$ in (i) above.
\end{enumerate}

Based on these, we have

\begin{theorem}\label{thm: continuous dependence}
Let $M$ be a  $d$-dimensional simply-connected
{differentiable} Riemannian manifold.
Let $E$ be a $\R^k$-vector bundle over $M$ with bundle metric $g^E$ and compatible bundle connection $\na^E$.
Denote by $\mathcal{A}(E)$ the space of affine connections on $E$ over $M$.
For $p>d$, define the spaces{\rm :}
$$
\mathbb{T}(M) := \left\{(g,B,\na^E)\,:\,
\begin{array}{ll}
g\in W^{1,p}(M;{\rm Sym}^+_{d\times d}{T^*M}),  \na^E\in L^p\big(M;\mathcal{A}(E)\big)\\
B\in L^p(M;{\rm Sym}_{d\times d}{T^*M}\otimes TM^\perp)\\
(g,B,\na^E) \text{ satisfies GCRE \eqref{gauss equation}--\eqref{ricci equation} on $E$}
\end{array}
\right\}
$$
and
\begin{equation}
\mathbf{V}(M) := W^{2,p}(M; \R^n) / {\rm Isom}_+(\R^{d+k}).
\end{equation}
Let $\Phi: \mathbb{T}(M) \map \mathbf{V}(M)$  map $(g,B,\na^E)$ to an isometric immersion
{$f:(M,g)\to(\R^{d+k},\euc)$} with second fundamental form $B$ and normal connection $\na^E$.
Then $\Phi$ is locally
{Lipschitz} continuous.
\end{theorem}

{In the above, ${\rm Sym}_{d\times d}T^*M$ denotes the space of symmetric $2$--forms
on the cotangent bundle $T^*M$ (which can be expressed as $d\times d$ symmetric matrices),
and ${\rm Sym}^+_{d\times d}T^*M$ consists of positive definite elements in ${\rm Sym}_{d\times d}T^*M$.
Also, it is classical (see \cite[p.56, Exercise 8(2)]{p}) that ${\rm Isom}_+(\R^{d+k}) = \R^{d+k}\rtimes \so(d+k)$;
that is,  orientation-preserving Euclidean rigid motions on $\R^{d+k}$
consist of translations and orientation-preserving rotations.
}\bigskip\\
\begin{proof}
{In this proof, denote by $\|\bullet\|_{\mathbb{T}}$ and $\|\bullet\|_{\mathbf{V}}$ the natural norms induced
from the suitable product or quotient topologies on $\mathbb{T}(M)$ and $\mathbf{V}(M)$, respectively.
More precisely, for any $h \in W^{1,p}(M;\mathfrak{gl}(d;\R))$, $D \in L^p(M;\mathfrak{gl}(d;\R))$,
and $\Lambda \in L^p(M;\mathcal{A}(E))$, we write
$$
\|(h,D,\Lambda)\|_{\mathbb{T}} := \|h\|_{W^{1,p}} + \|D\|_{L^p} + \|\Lambda\|_{L^p};
$$
and, for any $\vartheta \in \mathbf{V}(M)$, we write
$$
\|\vartheta\|_{\mathbf{V}}:=\inf_{\mathcal{I} \in {\rm Isom}_+(\R^{d+k})} \|\vartheta \circ \mathcal{I}\|_{W^{2,p}}.
$$
In general, arguments $h$ and $D$ in $\|(h,D,\Lambda)\|_{\mathbb{T}}$ above are not required to be tensorial.
Furthermore, in the above, the domains over which the Sobolev norms are taken are suitable subsets of $M$,
which will be clear from the context.}

We first use \cite[Theorem 5.2]{chenli} to conclude that
mapping $\Phi: \mathbb{T}(M) \map \mathbf{V}(M)$
is well-defined.
Then it suffices to establish
\begin{equation}\label{5.8a}
\|\Phi(g,B,\na^E)-\Phi(g',B',\na^{E'})\|_{{\mathbf{V}}} \leq C\|(g,B,\na^E)-(g',B',\na^{E'})\|_{{\mathbb{T}}}
\end{equation}
for some constant $C$, provided that the right-hand side is sufficiently small. The arguments are divided into three steps.

\smallskip
{\bf 1.}  The crucial step is to translate the Gauss--Codazzi--Ricci equations \eqref{gauss equation}--\eqref{ricci equation}
into the Pfaff--Poincar\'{e} system. This is achieved via the Cartan structural equations ({\it cf.} \cite{chern} for details).
Let $U\subset M$ be a local chart on which $E$ is trivialized.
For a local orthonormal frame of vector fields $\{\p_i\}_{i=1}^d \subset \G(TU)$,
let $\{\omega^i\}_{i=1}^d$ be the dual co-frame of differential $1$-forms.
In addition, let $\{\eta_\alpha\}_{\alpha=d+1}^{d+k}$ be an orthonormal frame on the typical fibre of $E$.
Then, for indices $1\leq i,j \leq d$, $\,d+1\leq \alpha,\beta \leq d+k\,$, and $1\leq a,b,c \leq d+k\,$,
define differential $1$-forms $\{\omega^a_b\}$ by
\begin{eqnarray}
&&\omega^i_j (\p_k) := \langle \na_{\p_k}\p_j, \p_i\rangle,\\
&& \omega^i_\alpha (\p_j) = -\omega^\alpha_i(\p_j): = \langle B(\p_i, \p_j), \eta_\alpha\rangle,\\
&& \omega^\alpha_\beta (\p_j) := \langle \na^E_{\p_j}\eta_\alpha,\eta_\beta\rangle.
\end{eqnarray}
We write $\mathbf{W} := \{\omega^a_b\}$ as the $\mathfrak{so}(d+k)$-valued differential $1$-forms on $U$,
or equivalently, the $1$-form-valued matrix field, where $\mathfrak{so}(d+k)$ is the space
of anti-symmetric $(d+k) \times (d+k)$ matrices. Schematically, we may write
$$
\mathbf{W} =
\begin{bmatrix}
\na&B\\
-B&\na^E
\end{bmatrix},
$$
where $\na=\na^g$ is the Levi-Civita connection on $M$ corresponding to metric $g$.
On the other hand, we augment the $1$-forms $\{\omega^i\}$ by setting
\begin{equation}\label{w}
w := \big(\omega^1,\omega^2,\cdots,\omega^d,\underbrace{0,\cdots,0}_{k \text{ times}}\big)^\top.
\end{equation}
Thus, $w$ is a $1$-form-valued $(d+k)$-vector, or a $\R^{d+k}$-valued $1$-form.

It can be checked
that the Gauss--Codazzi--Ricci equations \eqref{gauss equation}--\eqref{ricci equation}
are equivalent to the following two {\em structural equations} (
{see \cite[\S 5.3, Step~4]{chenli}}):
\begin{equation}\label{structural eqs}
\dd w = w \wedge \mathbf{W},\qquad \dd\mathbf{W} + \mathbf{W} \wedge \mathbf{W} =0,
\end{equation}
even for the lower regularity under consideration.
Moreover, the desired isometric immersion $f \in W^{2,p}(M;\R^{d+k})$ can be solved from the systems of
first-order nonlinear PDEs as below:
\begin{equation}\label{pfaff}
\mathbf{W} = \dd A \cdot A^\top,\qquad A(x_0) = A_0,
\end{equation}
and
\begin{equation}\label{poincare}
\dd f = w \cdot A, \qquad f(x_0)=f_0.
\end{equation}
In the above, $\dd$ is the exterior derivative, $\wedge$ denotes the intertwining of wedge product on differential forms
and the matrix product ($\cdot$), $A$ is a field of orthogonal $d+k$ matrices to be solved in some neighbourhood
$V \subset U \subset M$ containing point $x_0$,
 $A_0$ and $f_0$ are arbitrary initial data, \eqref{pfaff} is known as the {\em Pfaff} system,
 and \eqref{poincare} as the {\em Poincar\'{e}} system.
We also note that $\mathbf{W}$ is in $L^p$.

\smallskip
{\bf 2.} We now prove \eqref{5.8a} for the case that $g=g'$.
With the above preparation,
the analytic lemmas due to Mardare \cite{m05} may be directly applied.
By \cite[Theorem 7]{m05}, the Pfaff system \eqref{pfaff} has a unique solution $A\in W^{1,p}(V; \mathfrak{so}(d+k))$
if and only if a compatibility condition of involutiveness holds in the sense of distributions; see \cite[Eq.\,(5.12)]{chenli}.
Also, by \cite[Theorem 6.5]{m07} and a result by Schwartz \cite{schwartz},
the Poincar\'{e} system \eqref{poincare} has a unique solution $f \in W^{2,p}(V; \R^{d+k})$ if and only if a compatibility condition
of exactness holds in the sense of distributions; see \cite[Eq.\,(5.14)]{chenli}.
Furthermore, the solutions for the Pfaff and Poincar\'{e} systems depend continuously on the source term:
if $A$ and $A'$ are two solutions for \eqref{pfaff} with the same initial data associated with source terms $\mathbf{W}$ and $\mathbf{W}'$ respectively,
then
\begin{align}\label{A-A'}
\|A-A'\|_{W^{1,p}(V)}\leq C\|\mathbf{W}-\mathbf{W}'\|_{L^p(V)}.
\end{align}
If $f$ and $f'$ are two solutions to \eqref{poincare} with the same initial data, then
\begin{equation}\label{f-f'}
\|f-f'\|_{W^{2,p}(V)} \leq C\|w\cdot A-w'\cdot A'\|_{W^{1,p}(V)},
\end{equation}
where $C=C(d,k,p,V)$, and $V$ is a sufficiently small neighbourhood on $M$, \textit{i.e.}, these estimates are local.
In this case, $g=g'$ so that $w=w'$. Then we have
$$
\|f-f'\|_{W^{2,p}(V)} \leq C\|A-A'\|_{W^{1,p}(V)}.
$$
However, by well-known computations in differential geometry (see \cite[Steps 4--5, \S 5.3]{chenli} for details),
the structural equations \eqref{structural eqs} are equivalent to the aforementioned compatibility conditions
for the Pfaff and Poincar\'{e} systems, respectively.

As indicated earlier, $\Phi: \mathbb{T}(M) \map \mathbf{V}(M)$ is well-defined.
In \eqref{A-A'}--\eqref{f-f'}, we have proved that
$$
\|f-f'\|_{W^{2,p}(V)} \leq C\|\mathbf{W}-\mathbf{W}'\|_{L^p(V)},
$$
provided that the initial data coincide.
Without loss of generality, we can always assume the same initial data --- this can be achieved by applying a Euclidean
isometry in ${\rm Isom}_+(\R^{d+k})$, which is negligible by the quotient construction of $\mathbf{V}(M)$.
By the definition of $\mathbf{W}$, it is clear that
$$
\|\mathbf{W}-\mathbf{W}'\|_{L^p(V)} = \|B-B'\|_{L^p(V)} + \|\na^E - \na^{E'}\|_{L^p(V)}.
$$
Thus, the proof is complete on a local chart $V$ in the case that $g=g'$.
{The same holds if $V$ is replaced by $M$, thanks to the assumption that $M$ is simply-connected
and a monodromy argument. See Mardare  \cite[Theorems~6.4 and 6.5]{m07} for detailed arguments on simply-connected
Euclidean domains, which can be adapted to manifolds; also see Step 7 in \cite[\S 5.3]{chenli}.}

\smallskip
{\bf 3.}
It remains to prove \eqref{5.8a} for
{the general case $g \neq g'$.
It follows from the definition of $\mathbf{W}$ given in Step~1 that
$$
\|\mathbf{W}-\mathbf{W}'\|_{L^p(V)}\leq C \left\{\|\na^g -\na^{g'}\|_{L^p(V)} + \|B-B'\|_{L^p(V)} + \|\na^E-\na^{E'}\|_{L^p(V)}\right\},
$$
where $C$ is purely dimensional.}

{Note that $\|\na^g -\na^{g'}\|_{L^p(V)}$ is in turn controlled by $\|g-g'\|_{W^{1,p}(V)}$.}
Indeed, the coordinate-wise components of the Levi-Civita connection is given by the Christoffel symbols
$\G^k_{ij} = \frac{1}{2}g^{kl} \big\{\p_i g_{jl} + \p_j g_{il} - \p_l g_{ij}\big\}$,
where $g^{ij}=(g_{ij})^{-1}$.
As in $\S \ref{sec: convergence}$, recall that $\na = g^{-1} \star \p g$, where $\p$ denotes the (Euclidean) derivative.
Then
\begin{align*}
\|\na^g - \na^{g'}\|_{L^p(V)} =&\, \big\| \big(g^{-1}-(g')^{-1}\big)\star\p g + (g')^{-1} \star \p(g-g') \big\|_{L^p(V)}\\
\leq &\, \|g^{-1}(g-g')(g')^{-1}\|_{L^\infty(V)} \|\p g\|_{L^p(V)}\\
   &\, + \|(g')^{-1}\|_{L^\infty(V)} \|\p (g-g')\|_{L^p(V)},
\end{align*}
which is bounded by $C \|g-g'\|_{L^p(V)}$,
thanks to the Sobolev--Morrey embedding: $W^{1,p}_{\rm loc} (M) \emb L^\infty_{\rm loc}(M)$ for $p>d$,
where $C$ depends on the $W^{1,p}$-norm of $g$ and $g'$ (which is allowed as the theorem concerns only the local continuity of $\Phi$).
Thus, in view of Eq.\,\eqref{A-A'}, we have
\begin{align}\label{A-A', g neq g' case}
\|A-A'\|_{W^{1,p}(V)} &\leq C \|\mathbf{W}-\mathbf{W}'\|_{L^p(V)} \nonumber\\
&\leq C\|(g,B,\na^E)-(g',B',\na^{E'})\|_{\mathbb{T}}.
\end{align}
In particular,  $A \in W^{1,p}(V; \mathfrak{gl}(n,\R))$.
Then we consider the Poincar\'{e} system (as before, with the same initial data) associated to $g$ and $g'$:
\begin{equation*}
\dd f = w \cdot A, \qquad \dd f'=w'\cdot A',
\end{equation*}
where $w$ and $w'$ are defined as those in \eqref{w} with respect to $g$ and $g'$, respectively.
Since $w$ and $w'$ consist of the dual co-frames of the same orthonormal frame on $M$, then
\begin{equation}\label{w-w'}
\|w-w'\|_{W^{1,p}(V)} \leq C\|g-g'\|_{W^{1,p}(V)},
\end{equation}
where $C$ is a dimensional constant.
Thus, taking the difference of the two Poincar\'{e} systems, we see that
\begin{equation}
\dd (f-f') = (w-w')A + w(A-A').
\end{equation}
Then we have the following estimates:
\begin{align}\label{f-f', g neq g' case}
&\|f-f'\|_{W^{2,p}(V)} \nonumber\\
&\leq \|(w-w')A + w(A-A')\|_{W^{1,p}(V)} \nonumber\\
&\leq C\big\{\|A\|_{L^\infty(V)}\|w-w'\|_{W^{1,p}(V)} + \|A\|_{W^{1,p}(V)}\|w-w'\|_{L^\infty(V)} \nonumber\\
&\qquad\,\,\, + \|w\|_{L^\infty(V)}\|A-A'\|_{W^{1,p}(V)}  + \|w\|_{W^{1,p}(V)} \|A-A'\|_{L^\infty(V)}\big\}\nonumber\\
&\leq C\big\{\|A\|_{W^{1,p}(V)}\|w-w'\|_{W^{1,p}(V)} + \|w\|_{W^{1,p}(V)}\|A-A'\|_{W^{1,p}(V)}\big\}\nonumber\\
&\leq C'\big\{\|w-w'\|_{W^{1,p}(V)} + \|A-A'\|_{W^{1,p}(V)}\big\},
\end{align}
where we have used the estimates for the Poincar\'{e} system ({\it cf}. Mardare \cite{m07} and Schwartz \cite{schwartz}) in the first inequality,
the usual interpolation inequality in the second inequality, and  the Morrey--Sobolev embedding in the third inequality,
and constant $C'$ depends on $\|A\|_{W^{1,p}(V)}$ and $\|w\|_{W^{1,p}(V)}$, which again is allowed due to the locality of $\Phi$.

Therefore, combining Eqs.\,\eqref{A-A', g neq g' case}--\eqref{w-w'} with \eqref{f-f', g neq g' case} together,
we complete the proof.
\end{proof}

\medskip
Finally, let us revisit the weak rigidity theorem of isometric immersions established
in \cite{chenli} ({\it i.e.}, Proposition \ref{propn: chenli, weak rigid thm of isom imm}).
A slightly stronger version can be deduced --- In view of Theorem \ref{thm: continuous dependence},
we can drop the uniform $W^{2,p}_\loc$--boundedness of isometric immersions in the hypothesis:

\begin{corollary}
Let $M$ be a $d$-dimensional simply-connected Riemannian manifold with $W^{1,p}$-metric $g$ for $p>d$.
Let $\{\Phi^\e\}_{\e>0}$ be a family of isometric immersions of $(M,g)$ into
{the Euclidean space $\R^D$},
whose second fundamental forms and normal connections are $B^\e$ and $\na^{\e,\perp}$.
Assume that $B^\e$ and $\na^{\e,\perp}$ are uniformly bounded in $L^p_\loc$.
Then, modulo translations and rotations, $\Phi^\varepsilon$ converges weakly in $W^{2,p}_\loc$ to an isometric immersion $\Phi$ of $(M,g)$ and, more importantly,
the second fundamental form and normal connection of $\Phi$ are weak $L^p_\loc$-limits of $B^\e$ and $\na^{\e,\perp}$,
obeying the Gauss--Codazzi--Ricci equations \eqref{gauss equation}--\eqref{ricci equation}.
\end{corollary}

\medskip

\noindent
{\bf Acknowledgement}.
The authors would like to thank Tristan Giron and Cy Maor for helpful
discussions and remarks,
and are indebted to the anonymous referee for  meticulous reading and constructive suggestions.
The research of Gui-Qiang G. Chen  was supported in part by
the UK Engineering and Physical Sciences Research Council Awards
EP/L015811/1 and EP/V008854/1, and the Royal Society--Wolfson Research Merit Award WM090014 (UK).
The research of Siran Li was supported in part by the China Key Laboratory of Scientific and
Engineering Computing (Ministry of Education) Grant AF0710029/011, Shanghai Frontier Research Institute for Modern Analysis,
and the UK EPSRC Science and Innovation Award to
the Oxford Centre for Nonlinear PDE (EP/E035027/1). Part of the paper was completed
during his stay in Montr\'{e}al as a CRM--ISM postdoctoral fellow,
for which S. Li is indebted to the Centre de Recherches Math\'{e}matiques and the Institut des
Sciences Math\'{e}matiques for their hospitality.
M. Slemrod was supported in part by Simons Collaborative Research Grant 232531 from the Simons Foundation;
M. Slemrod also thanks the Oxford Centre for Nonlinear PDE for
the hospitality during his visits.

\end{document}